\documentclass[12pt,reqno]{amsart}
\usepackage{soul}
\usepackage{multicol,amsmath,graphics,cancel,soul,color,bbm,amsfonts,dsfont,tikz,mathrsfs,amssymb}
\usepackage[left=1in, right=1in, top=1.1in,bottom=1.1in]{geometry}
\usetikzlibrary{matrix,patterns,positioning}
\numberwithin{equation}{section}

\newcommand{\Gauss}{N}
\newcommand{\Prv}{M}
\newcommand{\St}{\mathcal{A}}
\newcommand{\target}{Z}
\newcommand{\Omss}{\mathfrak{Z}}
\newcommand{\Stp}{\mathcal{H}}

\newcommand{\Fc}{\mathcal{F}}

\newcommand{\Kc}{\mathcal{K}}

\newcommand{\Lc}{\mathcal{L}}

\newcommand{\Pc}{\mathcal{P}}

\newcommand{\C}{\mathbb{C}}

\newcommand{\E}{\mathbb{E}}

\newcommand{\N}{\mathbb{N}}
\newcommand{\Pb}{\mathbb{P}}

\newcommand{\R}{\mathbb{R}}

\newcommand{\X}{\mathbb{X}}

\newcommand{\nphi}{\psi}

\newcommand{\floor}[1]{\left\lfloor#1\right\rfloor}
\newcommand{\vertiii}[1]{{\left\vert\kern-0.25ex\left\vert\kern-0.25ex\left\vert #1
    \right\vert\kern-0.25ex\right\vert\kern-0.25ex\right\vert}}

\def\RR{\mathbb{R}}\def\R{\mathbb{R}}

\def\NN{\mathbb{N}}

\def\PP{\mathbb{P}}
\def\EE{\mathbb{E}}
\def\Var{\mathbb{V}\mathrm{ar}}

\def\al{\alpha}

\def\de{\delta}

\def\ep{\varepsilon}

\def\la{\lambda}

\def\wt{\widetilde}

\def\scr{\mathscr}
\def\dtv{d_\mathrm{TV}}
\def\dk{{d_\mathrm{K}}}
\def\dw{{d}_1}

\def\dh2l{\mathbf{d}_{\mathbb{H}_{2\ell}}}
\def\d2{\mathbf{d}_2}
\def\zetab{\bar\zeta}

\newcommand{\Indi}[1]{\mathbbm{1}({#1})}

\newtheorem{thm}{Theorem}[section]
\newtheorem{lemma}[thm]{Lemma}

\newtheorem{prop}[thm]{Proposition}

\newtheorem{theorem}{Theorem}[section]

\newtheorem{p}[theorem]{Proposition}

\theoremstyle{remark}

\theoremstyle{definition}

\newcounter{dummy} \numberwithin{dummy}{section}

\newtheorem{Theorem}[dummy]{Theorem}

\newtheorem{Lemma}[dummy]{Lemma}

\newtheorem{Remark}[dummy]{Remark}

\def\1{\mathbbm{1}}

\makeatletter
\@namedef{subjclassname@1991}{2020 Mathematics Subject Classification}
\makeatother

\begin{document}

\title[A probabilistic approach to the Erd\"os-Kac theorem for additive functions]{A probabilistic approach to the Erd\"os-Kac theorem for additive functions}
\author{Louis H. Y. Chen, Arturo Jaramillo, Xiaochuan Yang}
\address{Louis H. Y. Chen: Department of Mathematics, National University of Singapore, Block S17, 10 Lower Kent Ridge Road, Singapore 119076.}
\email{matchyl@nus.edu.sg}
\address{Arturo Jaramillo \& Xiaochuan Yang: Mathematics Research Unit, Universit\'e du Luxembourg, Maison du Nombre 6, Avenue de la Fonte, L-4364 Esch-sur-Alzette, Luxembourg.\\
Department of Mathematics National University of Singapore Block S17, 10 Lower Kent Ridge Road Singapore 119076.}
\email{arturo.jaramillogil@uni.lu}
\email{xiaochuan.j.yang@gmail.com}

\date{\today}


\begin{abstract}
We present a new perspective of assessing the rates of convergence to the Gaussian and Poisson distributions in the Erd\"os-Kac theorem for additive arithmetic functions $\nphi$ of a random integer $J_n$ uniformly distributed over $\{1,...,n\}$. Our approach is probabilistic, working directly on spaces of random variables without any use of Fourier analytic methods, and our $\nphi$ is more general than those considered in the literature. Our main results are (i) bounds on the Kolmogorov distance and Wasserstein distance between the distribution of the normalized $\nphi(J_n)$ and the standard Gaussian distribution, and (ii) bounds on the Kolmogorov distance and total variation distance between the distribution of $\nphi(J_n)$ and a Poisson
distribution under mild additional assumptions on $\nphi$. Our results generalize the existing ones in the literature. 
\end{abstract}
\maketitle
\section{Introduction}
\subsection{Overview}
The present manuscript aims to provide new probabilistic perspectives for the understanding of additive arithmetic
functions; namely, mappings $\nphi:\N\rightarrow\R$ satisfying the identity $\nphi(jk)=\nphi(j)+\nphi(k)$ when $j$ and $k$ are co-prime. 
Many of the functions of interest to number theorists are of this type, for example
the prime factor counting functions $\omega$ and $\Omega$, as well as the logarithm of any multiplicative function, such as
the sum of powers of divisors,  Euler's totient function, M\"obius function and Mangoldt's function (see \cite[Section~2.2]{Ten} for details).
The behavior of such functions is typically analyzed by counting the number of integers within a large interval, whose image under the
action of $\nphi$ lies in a given subset of $\mathbb{R}$. Probabilistically, this procedure is equivalent to
describing the action of $\nphi$ over a uniform random variable $J_n$ with uniform distribution over $\{1,\dots, n\}$.
The aim of this paper is to study the asymptotic behavior (as $n\rightarrow\infty$)
of the law of $\nphi(J_{n})$, with specific emphasis on the case where $\nphi$ is chosen to be the function that counts
the number of distinct prime factors.\\

\noindent It is worth mentioning that, although the probabilistic perspective for addressing this type of problems has been used for a long time
(see Section \ref{sec:tcopfcf} for details), only suboptimal descriptions for $\nphi(J_n)$ have been obtained probabilistically,
whereas the sharp ones have been derived by non-trivial number theoretical tools, such as Perron's formula, Dirichlet series and
estimates on the Riemann zeta function. The relevance of the present manuscript comes not only from the
main new results per se (which broadly speaking, can be described as ``sharp Gaussian and Poisson approximations for standardized and non-standardized versions
of $\nphi(J_n)$''), but also from the nature of the perspective, which
up to an estimation on the function $\pi(n)$ that counts the number of primes smaller than or equal to $n$, relies
almost entirely on probabilistic arguments and leads to conclusions as sharp as the ones obtained by
number theoretical tools, with a higher level of generality.\\

\noindent In order to set up an appropriate context for stating our main result, to be presented in detail in Section \ref{sec:main},
we first review briefly
some of the current literature related to limit theorems for $\nphi(J_{n})$. Our approach will be presented as part this literature, in
the section ``The conditioned independence approach''.
For convenience, we will assume that all the random variables throughout are defined in a common probability space
$(\Omss,\mathcal{F},\mathbb{P})$.

\subsection{The case of the prime factors counting function}\label{sec:tcopfcf}
Denote by $\mathcal{P}$ the set of prime numbers and by $[n]$ the set $\{1,\dots, n\}$.
One of the most important instances of additive functions for which the law of $\nphi(J_{n})$ can be successfully approximated, is the case where $\nphi$ is taken to be the mapping $\omega:\N\rightarrow\N$, defined by
\begin{align}\label{eq:littleomegadef}
\omega(k)
  &:=|\{p\in\Pc\ ;\ p \text{ divides } k\}|.
\end{align}
The value of $\omega(k)$ represents the number of prime divisors of a given integer $k\in\N$ without accounting for multiplicity. For instance, the value of $\omega(54)=\omega(2\times 3^2)$ is equal to two, since the only two primes that divide $54$ are two and three.\\

\noindent\textit{Classical Erd\"os-Kac theorem}\\
\noindent  The study of the distributional properties of $\omega(J_{n})$ began with the influential paper \cite{ErdKac1} by Paul Erd\"os and Mark Kac, where it was shown that the normalized random variables
\begin{align}\label{eq:omegaJnstand}
Z_n=\target_n^\omega	
  &:=\frac{\omega(J_n)-\log\log(n)}{\sqrt{\log\log(n)}}
\end{align}
converge in distribution towards a standard Gaussian random variable $\Gauss$. Since the publication of this result (nowadays known as the Erd\"os-Kac theorem), many improvements and developments on this topic have been considered. Among them is the paper \cite{Bill} by Billingsley, where the problem was addressed probabilistically by using the method of moments and the decomposition
\begin{align}\label{eq:omegadecomp1}
\omega(J_{n})
  &=\sum_{p\in\Pc\cap[n]}\Indi{p \text{ divides } J_n},
\end{align}
which simplifies the analysis of $\omega(J_{n})$ due to the fact that the summands on the right hand side can be shown to be asymptotically	independent.\\

\noindent The asymptotic Gaussianity of $Z_n$ raises the question of whether the associated convergence in distribution could be
quantitatively assessed with respect to a suitable probability metric, such as the Komogorov distance $\dk$ or the $1$-Wasserstein
distance $\dw$,  defined as
\begin{align*}
\dk(X,Y) = \sup_{z\in\RR} |\PP[X\le z] - \PP[Y\le z]|
\end{align*}
and
\begin{align*}
\dw(X,Y) = \sup_{h\in\mathrm{Lip}_{1}} |\EE[h(X)] - \EE[h(Y)]|,
\end{align*}
where $\mathrm{Lip}_1$ is the family of Lipschitz functions with Lipschitz constant at most one.
For this purpose, the idea of decomposing $\omega(J_{n})$ as a sum of random variables exhibiting a
``weak stochastic dependence''  is of great relevance from a probabilistic point of view, as it brings
the problem of studying $Z_{n}$ to the widely developed line of research of limit theorems for weakly
dependent sums of random variables; an area for which the powerful machinery of characteristic functions
and Stein's method is available. This idea has been exploited by many authors (see for instance \cite{TurRen},
\cite{LeVeq}, \cite{Kub}, \cite{Harp}, \cite{MRArratia}, \cite{BaKoNi}, \cite{JaKoNi} and \cite{KoNi}) who
have used a variety of techniques to find bounds for $\dk(Z_n,\Gauss)$. Next we present a brief summary  of
the main contributions to this topic.\\

\noindent\textit{LeVeque's conjecture}\\
The first assessment of the Kolmogorov distance between $Z_n$ and $\Gauss$ was presented in the paper \cite{LeVeq} by LeVeque, where it was shown that
$$\dk(\target_n,\Gauss)\leq C\frac{\log\log\log(n)}{\log\log(n)^{\frac{1}{4}}},$$
for some constant $C>0$ independent of $n$. In the same paper, it was also conjectured that the optimal rate was of the order $\log\log(n)^{-\frac{1}{2}}$; a claim that was subsequently shown to be true in the paper \cite{TurRen} by R\'enyi and Tur\'an. The approach presented in \cite{TurRen} relied on a careful study of the characteristic function of $\omega(J_n)$, based on Perron's formula, Dirichlet series, manipulations on contour integrals for analytic functions and some estimates on the Riemann zeta function $\zeta$ around the vertical strip $\{z\in\C\ ;\ \Re(z)=1\}$.\\

\noindent\textit{The Stein's method perspective}\\
Although the solution to LeVeque's conjecture presented by Tur\'an and R\'enyi in \cite{TurRen} is quite ingenious and beautiful, it is as well highly non-trivial from a probabilistic point of view. Moreover, up to this day, most of the approaches for obtaining bounds on $\dk(Z_n,\Gauss)$ (even those leading to suboptimal rates) are based on the analysis of the characteristic function of $\omega(J_{n})$, which requires deep and complicated number-theoretic manipulations. One of the alternative perspectives that have been proposed in the recent years, is the one presented in the paper \cite{Harp} by Harper, where techniques from Stein's method for weakly dependent random variables were used to prove that the truncated version of \eqref{eq:omegadecomp1},
\begin{align*}
V_n
  &:=\sum_{p\in\Pc\cap[n^{\frac{1}{3}\log\log(n)^{-2}}]}\Indi{p \text{ divides } J_n},
\end{align*}
satisfies $\dk(V_n,\Prv_n)\leq C\log\log(n)^{-1}$, where $\Prv_n$ is a Poisson random variable with intensity $\log\log(n)$ and $C>0$ is a universal constant independent $n$. The Poisson approximation approach presented in \cite{Harp} possesses two important features: in one hand, modulo a suitable estimation for the error of approximating $Z_n$ with $\log\log(n)^{-\frac{1}{2}}(V_n-\log\log(n))$, it provides an elementary approach for obtaining a  bound of the type
\begin{align}\label{eq:HarpboundK}
\dk(Z_n,\Gauss)\leq C\log\log\log(n)\log\log(n)^{-\frac{1}{2}},
\end{align}
where $C>0$ is an explicit constant. In addition to this, the fact that we can obtain a Poisson approximation for the law of $V_n$ is a phenomenon of great interest on its own, as the discrete nature of the Poisson distribution intuitively fits better that of $V_n$.

\noindent The idea of truncating the number of terms in \eqref{eq:omegadecomp1} was previously explored by Kubilius in \cite{Kub}, who proved among other things, a bound of the form \eqref{eq:HarpboundK} by means of an approximation of $V_n$ with a sum of \emph{fully independent} random variables. This result was subsequently sharpened by many authors (see \cite{BarVin}, \cite{Elliot}, \cite{Ten2}) and it is up to this day, a very useful tool for analyzing the law of $\psi(J_n)$ from a probabilistic perspective.\\

\noindent Both Harper's and Kubilius' approaches are very simple from a probabilistic point of view, but they have the disadvantage that the main contribution of the error in the estimation of $\dk(Z_n,\Gauss)$, comes from approximating the law of $\omega(J_{n})$ with $V_{n}$, and not from the approximation of $V_n$ with either a Poisson random variable (as in \cite{Harp}) or with a sum of independent random variables (as in \cite{Kub}). Thus, every analysis of the law of $\omega(J_n)$ that is based on a
description of $V_n$, regardless of the level of accuracy of the approximation of the law of $V_n$, can only lead to a bound of the form
$\dk(\omega(J_n),\Gauss)\leq C\log\log\log(n)\log\log(n)^{-\frac{1}{2}}$, which has a strictly slower asymptotic decay as the one from LeVeque's conjecture.\\

\noindent\textit{The mod-$\phi$ convergence perspective}\\
\noindent Recent developments on number theory have lead to a much better understanding of the characteristic function of $\omega(J_n)$ (see for instance \cite[Chapter~III.4]{Ten}), which has served as starting point for the heuristics of the papers \cite{JaKoNi}, \cite{KoNi} and \cite{BaKoNi}, where the powerful tool of mod-Gaussian and mod-Poisson convergence was developed and successfully applied to the analysis of the asymptotic law of $\omega(J_{n})$.
This type of technique, which we will refer to in the sequel as mod-$\phi$ convergence
(to avoid the specification on the Gaussian and Poissonian nature), aims to describe distributional properties of a given collection of
random variables $\{X_{n}\}_{n\in\N}$ by analyzing the quotient
\begin{align}\label{eq:modphiquo}
\frac{\E[e^{\textbf{i}\lambda X_n}]}{\E[e^{\textbf{i}\lambda U_n}]},
\end{align}
where $U_n$ is a random variable whose distribution is either a Gaussian or a Poisson. To exemplify the nature of \eqref{eq:modphiquo}, consider the case where $U_{n}$ is a standard Gaussian random variable and $X_{n}$ is an infinitely divisible random variable with unit Gaussian component and characteristic function
\begin{align*}
\E[e^{\textbf{i}\lambda X_n}]
  &=e^{\textbf{i}\mu_{n}-\frac{1}{2}\lambda^2+\int_{\R}(e^{\textbf{i}\lambda x}-1-\Indi{\{|x|<1\}}\textbf{i}\lambda x)\Pi_{n}(dx)},
\end{align*}
where $\mu_{n}\in\R$ and $\Pi_{n}$ is a Levy measure. For this instance, the quotient in \eqref{eq:modphiquo} takes the form
\begin{align*}
\frac{\E[e^{\textbf{i}\lambda X_n}]}{\E[e^{\textbf{i}\lambda U_n}]}
  =e^{\textbf{i}\mu_{n}+\int_{\R}(e^{\textbf{i}\lambda x}-1-\Indi{\{|x|<1\}}\textbf{i}\lambda x)\Pi_{n}(dx)},
\end{align*}
which is the characteristic function of the non-Gaussian part of $X_{n}$. In this sense, we can think of \eqref{eq:modphiquo} as a
quantity that describes the part of the characteristic function of $X_n$ that is not standard Gaussian (respectively, Poissonian). One should remark however, that for a more general random variable $X_{n}$, the quotient $\frac{\E[e^{\textbf{i}\lambda X_n}]}{\E[e^{\textbf{i}\lambda U_n}]}$ might not be the characteristic function of a random variable, which is an important observation to take into account when applying the heuristic above. Naturally, if the law of $U_n$ remains constant as $n$ varies, then the convergence of \eqref{eq:modphiquo} towards one implies the convergence in distribution of $X_n$ to $U_1$. However, a more complex phenomenology might appear in the case where $U_n$
varies and the aforementioned quotient converges to a non-constant limit. This idea was explored by Barbour, Kowalski and Nikeghbali,
where the distance between the laws of $X_{n}$ and a suitable Poisson random variable was described in terms of the regularity
properties of the limit of \eqref{eq:modphiquo} as a function of $\lambda$. These results were then applied to the case where
$X_{n}:=\omega(J_{n})$, for which a lot of information on the characteristic function of $\omega(J_{n})$ was available, leading
among other things, to the following remarkable result see \cite[Theorem~7.2]{BaKoNi}
\begin{Theorem}
There exists a constant $C>0$, such that
\begin{align}\label{dkModphi}
\dtv(\omega(J_n),\Prv_n)
  &\leq C\log\log(n)^{-\frac{1}{2}},
\end{align}
where $\Prv_n$ is a Poisson random variable with intensity
parameter $\log\log(n)$ and $\dtv(X,Y)$ denotes the total variation distance between two random variables $X$ and $Y$; namely,
\begin{align*}
\dtv(X,Y):= \sup_{A\in\mathcal{B}(\R)} | \PP[X\in A ] - \PP[Y\in A]|,
\end{align*}
where $\mathcal{B}(\R)$ the collection of Borel subsets of $\R$.

\end{Theorem}
\noindent Notice that as a corollary of \eqref{dkModphi}, one obtains an alternative proof of LeVeque's bound.
It is worth mentioning that in \cite{BaKoNi}, sharper approximations of the law of $\omega(J_n)$ were obtained by means of
Poisson-Charlier signed measures. The downside of applying the mod-$\phi$ convergence approach for proving LeVeque's conjecture,
is that prior knowledge on the characteristic function of $\omega(J_n)$ is required, which as in the paper of R\'enyi  and Tur\'an \cite{TurRen}, can only be obtained by means of analytic techniques from number theory.\\

\noindent\textit{The size-biased permutation approach}\\
In the paper \cite{MRArratia} by Arratia, an alternative probabilistic approach for studying the divisibility properties of $J_n$ was proposed.
The methodology consists in constructing a coupling of $J_n$ together with a partial product $T_n$ of size-biased permutated random primes,
in such a way that the total variation distance between $J_n$ and $T_nP_n$ is small, where $P_n$ is a suitable random prime
(see \cite[Section~1.2]{MRArratia}). Manipulations on the law of the size-biased permutation become tractable
after introducing a suitable point process with amenable independence properties (see \cite[Section~3.5]{MRArratia} for details).\\

\noindent Using the aforementioned ideas, it was shown in \cite[Theorem~3]{MRArratia} that if $d_{\mathsf{id}}:\N^2\rightarrow\N$ denotes the insertion deletion distance
\begin{align*}
d_{\mathsf{id}}(\prod_{p\in\Pc}p^{\alpha_p},\prod_{p\in\Pc}p^{\beta_p})
  &:=\sum_{p\in\Pc}|\alpha_p-\beta_p|,
\end{align*}
for $\prod_{p\in\Pc}p^{\alpha_p},\prod_{p\in\Pc}p^{\beta_p}\in\N$, and $d_{1,\mathsf{id}}$  the associated 1-Wasserstein distance
measure with respect to $d_{\mathsf{id}}$ , whose action over the laws of random variables $X,Y$ is given by
\begin{align*}
d_{1,\mathsf{id}}(X,Y) = \sup\{|\EE[h(X)] - \EE[h(Y)]|\ ;\ |h(x)-h(y)|\leq d_\mathsf{id}(x,y)\ \text{ for all } x,y\in\N\},
\end{align*}
then
\begin{align}\label{eq:limWassdelinser}
\lim_{n\rightarrow\infty}d_{1,\mathsf{id}}(J_n,\prod_{p\in\Pc\cap[n]}p^{\xi_p})
  &= 2,
\end{align}
where $\xi_p$ are independent Geometric random variables with $\Pb[\xi_p=k]=p^{-k}(1-p^{-1})$, for $k\geq0$.
This identity can be combined with classical results on sums of independent random variables in order to obtain a bound similar to that of LeVeque's conjecture,
but measured with respect to the 1-Wasserstein distance $d_{1}$. More precisely, it can be shown that \eqref{eq:limWassdelinser}
implies the existence of a constant $C>0$, such that
\begin{align}\label{eq:limWassdelinseraux}
d_{1}\big(Z_n,\Gauss\big)
  &\leq \frac{C}{\sqrt{\log\log(n)}}.
\end{align}
It is worth mentioning that although it is not clear how to use \eqref{eq:limWassdelinser} to get bounds of the type \eqref{eq:limWassdelinseraux} with respect to
the Kolmogorov distance $\dk$, the two-step approximation scheme of Arratia (approximating $J_n$ with $T_n$ and then $\omega(T_n)$ with
a Gaussian distribution) does inspire
us to find a  transparent divisibility structure in the intermediate step, a counterpart of his well elaborated $T_n$.
By doing so, we manage to find sharp bounds for both the Kolmogorov distance $\dk\big(Z_n,\Gauss\big)$ and the Wasserstein
distance $d_{1}\big(Z_n,\Gauss\big)$ (see Section \ref{sec:main} for details).\\

\noindent\textit{The function $\Omega$}\\
Another instance for which the law of $\nphi(J_n)$ can be suitably approximated, is the case where $\nphi$ is the prime factor counting function with multiplicity, defined by
\begin{align}\label{eq:bigomegadef}
\Omega(k)
  &:=\sum_{p\in\Pc\cap[k]}\max\{\alpha\geq0\ ;\ p^{\alpha}\ \text{ divides } k\}.
\end{align}
Unlike $\omega$, this function is not totally additive (meaning that $\Omega(p^{\alpha})$ doesn't necessarily coincide with $\Omega(p)$,
for $p\in\Pc$ and $\alpha\in\N$). However, most of the results related to the state of the art on the asymptotical distribution of $\omega(J_n)$ are also valid for $\Omega(J_n)$. In particular, the results presented by R\'enyi and Tur\'an in \cite{TurRen} and those presented by Barbour, Kowalski and Nikeghbali in \cite{BaKoNi} provide a bound of the type
\begin{align*}
\dk\bigg(\frac{\Omega(J_n)-\log\log(n)}{\sqrt{\log\log(n)}},\Gauss\bigg)
  &\leq C\log\log(n)^{-\frac{1}{2}},
\end{align*}
for some $C>0$. Moreover, the Poisson approximation (and Poisson-Charlier approximation) presented in \cite[Theorem~6.2]{BaKoNi}, establishes the bound
\begin{align}\label{dkModphibOm}
\dtv(\Omega(J_n),\Prv_n)
  &\leq C\log\log(n)^{-\frac{1}{2}},
\end{align}
where $C>0$ is a constant independent of $n$ and $\Prv_n$ is a Poisson random variable with parameter $\log\log(n)$. The paper \cite{Harp} doesn't explicitly state a Poisson approximation for $\Omega(J_n)$, although it is clear that the ideas from \cite{Harp} can be easily adapted to obtain a bound of the type $\dk(\tilde{V}_n,\Prv_n)\leq C\log\log(n)^{-1}$, where
\begin{align*}
\tilde{V}_n
  &:=\sum_{p\in\Pc\cap[1,n^{\frac{1}{3}\log\log(n)^{-2}}]}\max\{\alpha\ge 0 ;\ p^{\alpha} \text{ divides } J_n\}.
\end{align*}
The interested reader is also encouraged to see \cite[Section~5]{Harp} for an analysis of both $\omega(J_n)$ and $\Omega(J_n)$ via
exchangeable pairs. However, one should keep in mind that this approach leads to results strictly coarser than those from
\cite[Section~4]{Harp}.

\subsection{Erd\"os-Kac theorem for general additive functions}
The broad range of approaches and ideas nowadays available for addressing the classical Erd\"os Kac-theorem and LeVeque's conjecture, naturally brings the question of whether such techniques can be adapted to describe the asymptotic distribution of $\nphi(J_n)$ for a more general additive function $\nphi$. Although the convergence in distribution (without assessment on its Kolmogorov distance) has been known since the paper \cite{ErdKac1}, the adaptation of the proof of the optimal bounds obtained in the papers \cite{TurRen}, \cite{BaKoNi}, \cite{JaKoNi} and \cite{KoNi} to the general additive function case, is a surprisingly difficult task. This is mainly due to the fact that all the estimations of $\dk(Z_n,\Gauss)$ based on the use of characteristic function 
rely on the identity
\begin{align*}
\E[e^{\textbf{i}\lambda\omega(J_n)}]
 &=e^{\log\log(n)(e^{\textbf{i}\lambda}-1)}\exp\big\{\Gamma(e^{\textbf{i}\lambda})^{-1}\prod_{q\in\Pc}(1+q^{-1}(e^{\textbf{i}\lambda}-1))(1-q^{-1})^{e^{\textbf{i}\lambda}-1}+O(\log(n)^{-1})\big\},
\end{align*}
which doesn't necessarily hold when $\omega$ is replaced by $\psi$.
 Other perspectives, such as the one by Harper in \cite{Harp} and Kubilius in \cite{Kub} are quite versatile and extend easily to additive functions, but as mentioned before, they do not provide an optimal rate of convergence in the case for the prime factor counting functions $\omega$ and $\Omega$. This motivates the development of an alternative probabilistic tool that allows to optimally estimate
\begin{align}\label{eq:dKnormJnN}
\dk\big{(}\sigma_n^{-1}(\psi(J_n)-\mu_n),\Gauss\big{)},
\end{align}
where $\mu_n\in\R$ and $\sigma_n>0$ are such that $\sigma_n^{-1}(\psi(J_n)-\mu_n)$ converges in law to $\Gauss$.
The main goal for this paper consists in addressing the aforementioned problem from a probabilistic point of view,
relying as little as possible on sophisticated number theoretical arguments. In the sequel, we will refer to this methodology by
``the conditioned independence approach''.\\

\noindent\textit{The conditioned independence approach}\\
In the paper \cite{Harp} by Harper, it is mentioned that the decomposition \eqref{eq:omegadecomp1}, expressing $\omega(J_{n})$ as a sum
of weakly dependent random variables, suggests the use of the theory from Stein's method for estimating \eqref{eq:dKnormJnN}.
This idea partially influences
our approach, as we follow as well a Stein's method perspective. However, instead of viewing $\psi(J_n)$ as a sum of weakly dependent
random variables, we will use a two-step approximation strategy similar in spirit to those from \cite[Section~1.2]{MRArratia},
to show that, firstly, $\psi(J_n)$ is close in Kolmogorov distance to $\psi(H_n)$,
where $H_n$ is a random variable supported in $\{1,\dots, n\}$ and taking the value $k$ with probability proportional to $k^{-1}$.
We then carry out a Stein's method analysis over the variables $\psi(H_n)$, which surprisingly, is a considerably simpler task due to
a key identity in law (see Theorem \ref{thm:multiplicities}) which expresses the law of $\psi(H_n)$ as a sum of fully independent
random variables, conditioned on a suitable explicit event.\\

\noindent As one can expect, the conditioned independence provides a much
easier framework to apply probabilistic techniques, in comparison with the case of general weakly dependent random variables. We
utilize this neat structure to embed the underlying randomness of $\nphi(H_n)$ into a Poisson space,
which remarkably facilitates the application of Stein's method, in virtue of the celebrated
``Mecke's formula''.
This embedding procedure is entirely different from previous
approaches and consists on comparing the behavior of $\nphi(H_n)$ with that of a random variable of the form $\tilde{\nphi}(H_n)$, where
$\tilde{\nphi}$ is a additive arithmetic function characterized by the identity
$\tilde{\nphi}(p^{\alpha})=\alpha\nphi(p)$, valid for all $p\in\Pc$ and $\alpha\in\N_0$.
We would like to refer the reader to \cite{MRArratia} for a Poissonian embedding of the randomness of
the prime decomposition of $J_n$, based not on independent random variables conditioned on a certain constraint, but rather on a beautiful
parallelism with uniform random permutations.
It is interesting to remark that the approximating function $\tilde{\nphi}$ that we utilize doesn't satisfy the property
$\tilde{\nphi}(p^{\alpha})=\tilde{\nphi}(p)$ (namely, it is not ``totally additive'' 
the classic heuristics that additive functions are easier to study when we approximate them with totally additive functions
(this is the case of the analysis of $\Omega(J_n)$, which is obtained from properties of $\omega(J_n)$).\\

%

\noindent Before elaborating further on the details of this methodology, we will introduce some notation and establish basic
assumptions on $\psi$. We will write $\N_0$ to denote the set of natural numbers including zero, namely $\N_0:=\N\cup\{0\}$.
The probability law of a random variable $X$, will be denoted by $\mathcal{L}(X)$.
For a given $p\in\Pc$, consider the $p$-adic valuation function $\alpha_p:\N\rightarrow\N$, defined as the unique mapping satisfying the prime factorization
\begin{align}\label{eq:alphadef}
k=\prod_{p\in\Pc}p^{\alpha_p(k)},
\end{align}
for all $k\in\N$. It is plain that if $k\leq n$, one only needs to consider $p\in\mathcal P_n$ in the factorization where
$$\mathcal P_n:=\mathcal P\cap [n].$$
Let $\{\xi_p\}_{p\in\Pc}$ be a collection of independent gemetric random variables with
\begin{align*}
\Pb[\xi_p=k]
  &=(1-p^{-1})p^{-k},
\end{align*}
for all $k\in\N_{0}$ and $p\in\mathcal P$. Our main result requires the following assumptions.

\begin{enumerate}
\item[\textbf{(H1)}] The function $\nphi$ restricted to $\Pc$ is bounded. Namely,
\begin{align*}
c_1
  &:=\sup_{p\in\Pc}|\nphi(p)|<\infty.
\end{align*}

\item[\textbf{(H2)}] Suppose
\begin{align*}
c_2
  &:=\left(
\sum_{p\in\Pc}\frac{\EE[\psi(p^{\xi_p+2})^2]}{p^2}\right)^{1/2}
  <\infty,
\end{align*}
which can be shown to be equivalent to
\begin{align*}
\sum_{p\in\Pc} \sum_{k\ge 2} \frac{\psi(p^k)^2}{p^k} <\infty.
\end{align*}
\end{enumerate}

\noindent Since the type of result that we are seeking is asymptotic as $n$ approaches infinity, we will assume in the sequel
that $n\geq 21$.  We use the notation $\zetab(s)$ to denote Riemann's zeta function minus 1, namely $\sum_{k\ge 2} k^{-s}$ which
serves as upper bound for the series $\sum_{p\in\Pc} p^{-s}$ with some specific choices of $s$.

\section{Main Results}\label{sec:main}
\noindent
In this section we present our main results. Let $\mu_{n}$ and $\sigma_{n}>0$ be given by
\begin{align}\label{eq:musigdef}
\mu_{n}
  =\sum_{p\in\Pc_n} \nphi(p)p^{-1}(1-p^{-1})^{-1}
	\ \ \ \ \ \ \text{ and }
	\ \ \ \ \ \
\sigma_{n}^2
  =\sum_{p\in\Pc_n}\nphi(p)^2p^{-1}(1-p^{-1})^{-2}.
\end{align}
We will assume without loss of generality that $\nphi$ is not-identically zero and  $n\geq 27$
is sufficiently large so that $\sigma_n$ is strictly positive.
Our main results are the following bounds.
\begin{Theorem}\label{thm:mainunif}
Suppose that $\nphi$ satisfies \textbf{(H1)} and \textbf{(H2)}.  Then, provided that $\sigma_{n}^2\geq 3(c_1^2+c_2^2)$,
\begin{align}\label{eq:Kolmogorov_Jn}
\dk\left(\frac{\nphi(J_{n})-\mu_n}{\sigma_n},\Gauss\right)
  &\leq \frac{\kappa_1}{\sigma_n}+\frac{\kappa_2}{\sigma_n^2}+\frac{\kappa_3\log\log(n)}{\log(n)},
\end{align}
and
\begin{align}\label{eq:Wasserstein_Jn}
d_1\left(\frac{\nphi(J_{n})-\mu_n}{\sigma_n},\Gauss\right)
  &\leq \frac{\kappa_4}{\sigma_n}+\kappa_5\frac{\log\log (n)^{\frac{3}{2}}}{\log (n)^{\frac{1}{2}}},
\end{align}
where $\Gauss$ is a standard Gaussian random variable and
\begin{align}\label{eq:kappadef}
\begin{array}{lll}
\kappa_1:=65c_1+66c_2  &  \kappa_2:=726c_1^2+116c_1c_2&\kappa_3:=67.4\\
\kappa_4:=106c_1 +  2c_2  &  \kappa_5:= 49.3.
\end{array}
\end{align}
\end{Theorem}

\begin{Remark}\label{eq:Remone}
As one can observe from the proof of Theorem \ref{thm:mainunif} (to be presented in Section \ref{mainthmWboundCLT}-\ref{sec:mainunifthm}),
the use of the normalizations $\mu_n$ and $\sigma_n$ is quite natural from a probabilistic perspective, as they represent the mean
and variance of an approximating sum of independent random variables. However, one should keep in mind
that any choice of asymptotically equivalent normalizing constants leads to an equivalent version of Theorem \ref{thm:mainunif},
provided that we
suitably modify the bounds. More precisely, let $m_n$ and $s_n$ be another pair of normalizing constants. Then it is simple to check that
\begin{align*}
\dk\Big(\frac{\psi(J_n)-m_n}{s_n}, N\Big)\le \dk\Big(\frac{\psi(J_n)-\mu_n}{\sigma_n}, N\Big) + \dk(N, W),
\end{align*}
where $W$ is a normal random variable with mean $(\mu_n-m_n)/s_n$ and variance $\sigma_n^2/s_n^2$. Similarly,  the relation $\dw(aX+c,aY+c)=|a|\dw(X,Y)$ for $a,c\in\RR$ and arbitrary random variables $X,Y$ implies
\begin{align*}
\dw\Big(\frac{\psi(J_n)-m_n}{s_n}, N\Big)\le \frac{\sigma_n}{s_n} \dw\Big(\frac{\psi(J_n)-\mu_n}{\sigma_n}, N\Big) + \frac{\sigma_n}{s_n}\dw(N, W'),
\end{align*}
where $W'$ is a normal random variable with mean $(m_n-\mu_n)/\sigma_n$ and variance $s_n^2/\sigma_n^2$. The additional error term in both bounds is simple to estimate, see Lemma \ref{l:GaussianDistance}.
\end{Remark}

\begin{Remark}\label{eq:Remone2}
We work out the concrete example of $\psi=\omega$ with normalizing constants $m_n=s_n^2=\log\log(n)\le \mu_n \le \sigma_n^2$. By Theorem \ref{thm:mainunif}, Lemma \ref{l:GaussianDistance} and  Mertens' formula (the two-sided bounds  \eqref{eq:Mertensinv}), we have
\begin{align*}
\dk\Big(\frac{\omega(J_n)-\log\log(n)}{\sqrt{\log\log(n)}}, N\Big)\le \frac{118.9}{\sqrt{\log\log(n)}} + \frac{823.1}{\log\log(n)} + \frac{67.4\log\log(n)}{\log(n)}\le \frac{599}{\sqrt{\log\log(n)}},
\end{align*}
which gives a quantitative version of LeVeque's conjecture with explicit constant. We leave the extension of the above argument to $\psi=\Omega$ and bounds in $\dw$ to the interested reader.
\end{Remark}

\noindent As a byproduct of our analysis, we will obtain as well analogous theorems for the case where the  $J_{n}$ are replaced by random variables supported in $[n]$ and taking the value $k$ with probability proportional to
$k^{-1}$ for $k\in\N$.

\begin{Theorem}\label{thm:CLTHarmonic}
Suppose that $\nphi$ satisfies \textbf{(H1)} and \textbf{(H2)}. Let $\{H_{n}\}_{n\geq 1}$ be a sequence of random variables supported in $\N\cap[n]$, with
\begin{align*}
\Pb[H_{n}=k]
  &=\frac{1}{L_nk},
\end{align*}
for $k\in\{1,\dots, n\}$, where $L_n:=\sum_{j=1}^{n}\frac{1}{j}$. Then, provided that $\sigma_{n}^2\geq 3(c_1^2+c_2^2)$,
\begin{align}\label{eq:Kolmogorov}
\dk\left(\frac{\nphi(H_{n})-\mu_n}{\sigma_n},\Gauss\right)
  &\leq \frac{\gamma_{1}}{\sigma_n} + \frac{\gamma_2}{\sigma_n^2},
\end{align}
and
\begin{align}\label{eq:Wasserstein}
\dw\left(\frac{\nphi(H_{n})-\mu_n}{\sigma_n},\Gauss\right)\le \frac{\gamma_3}{\sigma_n},
\end{align}
where $\mathcal{N}$ is a standard Gaussian random variable and
\begin{align}\label{eq:cidef}
\gamma_{1}  := 32c_1+33c_2,  \  \ \ \gamma_2:= 363c_1^2+58c_1c_2,
\ \  \ \ \ \ \
\gamma_{3}  :=105c_1+2c_2.
\end{align}
\end{Theorem}
\medskip

\noindent For $\NN_0$-valued additive functions, the following Poissonian approximations can be proved.

\begin{Theorem}\label{thm:PLTHarmonic}
Assume that $\psi:\NN\to \NN_0$ and
\begin{align}\label{e:positivity}
\lambda_n:=\sum_{p\in\Pc_n}\frac{\nphi(p)}{p-1}>0, \quad n\in\NN.
\end{align}
Let $\Prv_n$ be a Poisson random variable with intensity $\lambda_n$ and
suppose that $\nphi$ satisfies \textbf{(H1)} and \textbf{(H2)}. Then,
\begin{align}\label{eq:citildedef}
\dtv(\psi(H_n),\Prv_n)
  &\le \frac{\tilde{\gamma}_{1}}{\sqrt{\la_n}} +  \frac{\tilde{\gamma}_{2}}{\la_n} + \frac{2c_1}{\lambda_n}\sum_{p\in\Pc_n} \frac{|\psi(p)-1|}{p}.
\end{align}
where
\begin{align*}
\tilde{\gamma}_{1}  := 17c_1+2c_2,  \  \ \ \tilde{\gamma}_2:= 2.4c_1^2+8.2c_1c_2+4c_1.
\end{align*}
In particular, if $\nphi(p)=1$ for all $p\in\Pc$, then
\begin{align}\label{eq:PLTharmonic}
\dtv(\psi(H_n),\Prv_n)
  &\leq \frac{10}{\sqrt{\la_n}} + \frac{6.4 + 8.2c_2}{\la_n}.
\end{align}
\end{Theorem}

\begin{Theorem}\label{t:poisson} Suppose that $\nphi: \NN\to\NN_0$ satisfies \textbf{(H1)}, \textbf{(H2)} and \eqref{e:positivity}.
\begin{itemize}
\item[(i)] We have
 \begin{align}\label{eq:t:poisson}
\dk(\psi(J_n),M_n)
\le \frac{\tilde{\kappa}_1}{\sqrt{\la_n}} +  \frac{\tilde{\kappa}_2}{\la_n} +
\frac{4c_1}{\lambda_n}\sum_{p\in\Pc_n} \frac{|\psi(p)-1|}{p} + \kappa_3\frac{\log\log(n)}{\log(n)}.
\end{align}
where
\begin{align}\label{eq:kappatildedef}
\begin{array}{lll}
\tilde{\kappa}_1:=
51c_1+6c_2+1 &
\tilde{\kappa}_2:=
7.2c_1^2+24.6c_1c_2+12c_1+2.4(c_1\vee 1)&\kappa_3:=67.4\\
\end{array}
\end{align}
\item[(ii)] Assume further that $\psi(p)=1$ for all $p\in\Pc$.  Then
\begin{align*}
\dtv(\psi(J_n),M_n)\le \frac{18+2c_2}{\sqrt{\la_n}} +  \frac{6.4+8.2c_2}{\la_n} + \kappa_3\frac{\log\log(n)}{\log(n)}.
\end{align*}
\end{itemize}
\end{Theorem}

\begin{Remark}
Provided that the bounds from Theorems \ref{thm:PLTHarmonic} and \ref{t:poisson} are of the order $\lambda_n^{-\frac{1}{2}}$,
we can obtain an alternative approach for proving Theorems \ref{thm:mainunif} and \ref{thm:CLTHarmonic}; as one can first approximate
the law of $\nphi(J_{n})$ (respectively $\nphi(H_{n})$) with a Poisson distribution, and then the normalized Poisson distribution
with a standard Gaussian law. However, one should keep in mind that
for a large family of arithmetic additive functions $\nphi$, the term
\begin{align*}
\frac{1}{\lambda_n}\sum_{p\in\Pc_n} \frac{|\psi(p)-1|}{p}
\end{align*}
might not converge to zero, which will prevent us from obtaining Gaussian approximations from Theorems \ref{thm:PLTHarmonic} and \ref{t:poisson}.
\end{Remark}

\begin{Remark}
Up to a multiplicative constant independent of $n$, the bound from Theorem \ref{t:poisson}  implies the one presented in the paper \cite[Theorem~7.2]{BaKoNi}.
\end{Remark}

\noindent One should observe that Remarks \ref{eq:Remone} and \ref{eq:Remone2} apply as well to Theorems \ref{thm:CLTHarmonic}-\ref{t:poisson}
for a suitable modification of the upper bounds appearing therein.\\

\noindent The rest of the paper is organized as follows. In Section \ref{sec:prelim} we present some useful preliminaries on
number theoretical results, divisibility properties of $\Lc(J_{n})$, Stein's method and integration by parts for Poisson functionals.
In sections \ref{mainthmWboundCLT}-\ref{sec:mainunifthm}
we present the proofs of Theorem \ref{thm:CLTHarmonic} and Theorem \ref{thm:mainunif}. The Poisson approximation results from
Theorems \ref{t:poisson} and \ref{thm:PLTHarmonic} are proved in Section \ref{sec:theorempoisson}. Finally, in the appendix we state and prove a
generalized version of Lemma \ref{thm:multiplicities} (whose elementary version plays a fundamental role in our methodology), as well
as some useful estimates.

\section{Auxiliary results}\label{sec:prelim}
\subsection{Elementary results from number theory}\label{Sec:ReultsNT}
We present the number theoretic results that will be required for our computations.
With the exception of the estimations on the prime counting function $\pi:\N\rightarrow\N$, all of these results can
be proved fairly easily.\\

\noindent\textit{Prime counting function inequalities}\\
Denote by $\pi:[1,\infty)\rightarrow\N$ the prime counting function, defined by
\begin{align*}
\pi(x)
  &:=|\Pc\cap[1,x]|.
\end{align*}
The existence of infinitely many primes implies that $\pi(n)$ converges to infinity as $n\rightarrow\infty$. There have been many efforts
to address the highly non-trivial task of describing as sharply as possible the asymptotic behavior of this function. The interested reader
is referred to the book \cite{Ten} for a historical compendium of some of the most popular methods that have been used to achieve this goal. In this
manuscript, we will use two results related to this problem, which we state next: for every $n\geq 1$, we have that
$$\pi(n)\leq\frac{1.5 n}{\log(n)}.$$
Furtheremore, if $n\geq17$, then
\begin{align}\label{eq:pinlowerb}
\pi(n)
  &\geq\frac{n}{\log(n)},
\end{align}
and if $n\geq229$,
\begin{align}\label{eq:Trudgian}
\left|\pi(n)-\int_0^{n}\frac{1}{\log(t)}dt\right|
  &\leq \frac{0.2795n}{\log(n)^{\frac{3}{4}}}\exp\bigg{\{}-\sqrt{\frac{\log(n)}{6.455}}\bigg{\}}
	\leq \frac{181n}{\log(n)^{3}},
\end{align}
where in the last inequality we used the fact that $x^{4.5}e^{-x}\leq 9.7$. In particular, since $|\int_{0}^n\frac{1}{\log(t)}dt-n\log(n)^{-1}-n\log(n)^{-2}|\leq 3n\log(n)^{-3}$,
\begin{align}\label{eq:Trudgian2}
|\pi(n)-\frac{n}{\log(n)}-\frac{n}{\log(n)^2}|
  &\leq \frac{184n}{\log(n)^3}.
\end{align}
The proofs of \eqref{eq:pinlowerb} and \eqref{eq:Trudgian} can be found in \cite{Rosser} and \cite{Trud}, respectively.\\

\noindent \textit{Rosser and Schoenfeld inequalities}\\
We will require as well suitable bounds for $\prod_{p\in\Pc_n}(1-p^{-1})$. The bound that we will use was first presented in the paper \cite{Rosser}
(see as well \cite[page 17]{Ten}). There exists a function $g:\R_{+}\rightarrow\R$, such that for all $n\in\N$, $|g(n)\log(n)|\leq 2$ and
\begin{align*}
\prod_{p\in\Pc_n}(1-p^{-1})
  &=\frac{e^{-\gamma}}{\log(n)}e^{g(n)},
\end{align*}
where $\gamma\approx0.577$ is Euler's constant. In addition, we have that
\begin{align}\label{bound:primeprod}
\prod_{p\in\Pc_n}(1-p^{-1})
  &>e^{-\gamma} \log(n)^{-1}(1-\log(n)^{-2}),
\end{align}
where $\gamma\approx0.577$ is Euler's constant, see \cite[page 17]{Ten}.\\

\noindent\textit{Divisibility probabilities for }$J_{n}$\\
Throghout this manuscript, we will repeadetely use the fact that the probability that a given positive integer $d\in\N$ divides
the random variable $J_n$ can be expressed as
\begin{align}\label{eq:probdivide}
\Pb[d\text{ divides } J_{n}]
  &=\frac{1}{n}\sum_{k=1}^n\Indi{d\ \text{  divides }\ k}
	=\frac{1}{n}\left\lfloor\frac{n}{d}\right\rfloor.
\end{align}

\noindent\textit{Mertens' formulas}\\
It is a well-known, elementary result from number theory, that sums of the form
$\sum_{p\in\Pc_n}\frac{\log(p)}{p}$ and $\sum_{p\in\Pc_n}\frac{1}{p}$, with  $n\in\N$, can be easily estimated as described below.
Such results are attributed to Franz Mertens
\begin{align}\label{eq:Mertenslog}
\log(n)-2
  &\leq\sum_{p\in\Pc_n}\frac{\log(p)}{p}<\log(n).
\end{align}
The proof is given in \cite[page 14]{Ten}. In addition, there exists a constant $C>0$, with $C\approx0.261$, such that
\begin{align}\label{eq:Mertensinv}
\log\log(n)
  &\leq \sum_{p\in\Pc_n}\frac{1}{p}
  \leq\log\log(n)+C+\frac{2}{\log(n)}.
\end{align}
For a proof, see for instance \cite[page 15]{Ten}.

\subsection{Approximating $\mathcal{L}(J_n)$ with $\Lc(H_n)$}\label{Sec:Arratiatype}
In this section we present a link between the probability laws of the random variables $J_n$ and $H_n$. To achieve this, we will make
use of ideas that are close in spirit to those from \cite[Sections~1.2~and~3.6]{MRArratia}.
Let $\{H_{n}\}_{n\geq 1}$ be given as before.
Let $\{Q(k)\}_{k\geq 1}$ be a sequence of random variables defined in $(\Omega,\Fc,\Pb)$, independent of $(J_n,H_n)$ and satisfying
the property that $Q(k)$ has uniform distribution over the set
\begin{align}\label{eq:Pckstardef}
\Pc_k^*
  &:=\{1\}\cup\Pc_k.
\end{align}
Namely,
\begin{align*}
\Pb[Q(k)=j]
  &=\frac{1}{\pi(k)+1},
\end{align*}
for $j\in\{1\}\cup\Pc_k$.
\begin{Lemma}\label{lem:arra}
Let $J_n, H_n$ and $\{Q(k)\}_{k\geq 1}$ be as before. Then, for $n\ge 21$, we have
\begin{align*}
\dtv(J_n, H_n Q(n/H_n))\le 61\frac{\log\log n}{\log n}.
\end{align*}
\end{Lemma}
\begin{proof}
For each $m\in[n]$, one has
\begin{align*}
\PP[H_n Q(n/H_n)=m]
  &= \sum_{\substack{p\in\Pc_n^{*}\\p|m}} \PP\Big[H_n= m/p, Q(\floor{np/m}) = p\Big] =\frac{1}{nL_n}\sum_{\substack{p\in\Pc_n^{*}\\p|m}}\frac{np/m}{1+\pi( np/m )}.
\end{align*}
Then,
\begin{align*}
|\PP[H_n Q(n/H_n)=m]  - \PP[J_n=m]|\le \frac{1}{nL_n}| \sum_{\substack{p\in\Pc_n^{*}\\p|m}} (\log(np/m)-1)  - L_n | + \frac{K}{nL_n}(1+\omega(m)),
\end{align*}
where
\begin{align*}
K:=\sup_{x\geq 1}\Big|\frac{x}{1+\pi(x)} - (\log x -1)\Big|.
\end{align*}
Therefore, setting $s(m):=\prod_{p|m}p$, one has
\begin{align}\label{eq:Arratiabound}
&\dtv(J_n,H_nQ(n/H_n)) = \frac{1}{2}\sum_{m=1}^n |\PP[H_nQ(n/H_n)=m]-\PP[J_n=m]|\leq R_1+R_2,
\end{align}
where
\begin{align*}
R_1
  &:=\frac{1}{2nL_n}|\sum_{m=1}^n (1+\omega(m))(\log(n/m)-1) +  \log s(m) - L_n|,\\
R_{2}
  &:=\frac{K}{2nL_n}\sum_{m=1}^n(1+\omega(m)).
\end{align*}
To bound $R_1$, we use the fact that $s(m)\leq\log(n)<L_n$ for all $m\leq n$, to deduce that
\begin{align*}
R_{1}
  &\leq \frac{1}{2L_n} \EE[|\log(n/J_n)-1|]
	+\frac{1}{2L_n} \EE[|\omega(J_n)(\log(n/J_n)-1)|]+\frac{1}{2L_n} |\EE[L_n - \log s(J_n)]|.
\end{align*}
By applying Cauchy-Schwarz inequality to the first two terms, we deduce that $R_{1}$ is bounded from above by $R_{1,1}+R_{1,2}$, where
\begin{align*}
R_{1,1}
  &:= \frac{1}{2L_n} (\EE[\omega(J_n)^{2}]^{\frac{1}{2}}+1)\EE[| \log(n/J_n)-1 |^{2}]^{\frac{1}{2}}\\
R_{1,2}
  &:=\frac{1}{2L_n}|\EE[L_n - \log s(J_n)]|.
\end{align*}
The term $R_{1,1}$ can be bounded by using the integral approximation
\begin{align*}
\EE[(\log(n/J_n) -1 )^2]
  &= 1+\frac{1}{n}\sum_{m=1}^{n-1}\log(n/m)^2-\frac{2}{n}\sum_{m=1}^{n-1}\log(n/m)\\
	&\leq  1+\sum_{m=1}^{n-1}\int_{\frac{m-1}{n}}^{\frac{m}{n}}\log(1/x)^2dx-2\sum_{m=1}^{n-1}\int_{\frac{m}{n}}^{\frac{m+1}{n}}\log(1/x)dx\\
	&\leq 1+2\int_{0}^{\frac{1}{n}}\log(1/x)dx,
\end{align*}
where the last step follows from the fact that $\int_{0}^{1}\log(x)^2dx=2$ and $\int_{0}^{1}\log(1/x)dx=1$. Thus, using the condition
$n\geq 21$, we obtain
\begin{align*}
\EE[(\log(n/J_n) -1 )^2]
  &\leq 1.4.
\end{align*}
As a consequence, by \eqref{eq:Eomegasqu}, we have
$$R_{1,1} \le 2\frac{\log\log n}{L_n}.$$
On the other hand, it is easy to see by integral approximation that
\begin{align*}
\log(n)-1
  &=\frac{1}{n}\int_{0}^{n}\log(x)dx
	\leq \frac{1}{n}\sum_{m=1}^n\log(m)
	\leq \frac{1}{n}\int_{1}^{n+1}\log(x)dx=\frac{n+1}{n}\log(n+1)-1,
\end{align*}
which, combined with the relation $\log(n+1)\leq L_n\leq \log(n)+1$, leads to

\begin{align*}
|L_n-\EE [\log (J_n)]|\leq 2.
\end{align*}
On the other hand,
\begin{align*}
\EE[\log(J_n) - \log (s(J_n))] &=\sum_{p\in\mathcal P_n} \EE[(\al_p(J_n)-1) \mathbbm{1}(\al_p(J_n)\ge 2)] \log (p) \\
&= \sum_{p\in\mathcal P_n} \EE[(\al_p(J_n)-1)_+] \log (p)\\
&\le \sum_{p\in\mathcal P_n} \log (p)\sum_{k\ge 1} \PP[\al_p(J_n)\ge k] = \sum_{p\in\mathcal P_n} \frac{\log (p)}{p^2}(1+(1-p^{-1})^{-1})\le 3,
\end{align*}
so that $R_{1,2}\leq\frac{5}{2L_n}$. We thus conclude that
\begin{align}\label{eq:R1arratiaboundf}
R_{1}\leq \frac{4.5\log\log (n)}{L_n}\leq \frac{4.5\log\log(n)}{\log(n)}.
\end{align}
\noindent To bound $R_2$, we use \eqref{eq:Eomegasqu} as well as the condition $n\geq 21$, to show that
$$\frac{1}{2n}\sum_{m=1}^n(1+\omega(m))\leq \frac{1}{2}\log\log(n)(1+\frac{2.5}{\log\log(n)})\leq 1.63\log\log(n),$$
which leads to
\begin{align}\label{eq:R2boundprevf}
R_{2}
  &\leq \frac{1.63K}{L_n}\log\log(n)\leq\frac{1.63K\log\log(n)}{\log(n)}.
\end{align}
It thus remains to bound $K$. To this end, we notice that by \eqref{eq:pinlowerb} and \eqref{eq:Trudgian2}, for every $x\geq 229$,
\begin{align*}
\Big|\frac{x}{1+\pi(x)} - (\log(x) -1)\Big|
  &\leq \frac{\log(x)}{x}\Big|x -(1+\pi(x)) (\log(x) -1)\Big|\\
	&\leq \frac{\log(x)}{x}\Big|x -(1+\frac{x}{\log(x)}+\frac{x}{\log(x)^2}) (\log(x) -1)\Big|
	+\frac{184}{\log(x)}\\
	&= \frac{\log(x)}{x}\Big|1+\frac{x}{\log(x)^2}-\log (x)\Big|
	+\frac{184}{\log(x)},
\end{align*}
so that
\begin{align*}
\Big|\frac{x}{1+\pi(x)} - (\log(x) -1)\Big|
  &\leq \frac{185}{\log(x)}\le 34.1.
\end{align*}
For $1\le x\le 229$, by applying the bound
\begin{align*}
\left|\frac{x}{1+\pi(x)} - (\log(x)-1)\right|
  &\leq \left|\frac{x}{1+\frac{x}{\log(x)}}\right| + \left|\log(x)-1\right|
\le 10
\end{align*}
for $x\geq 17$ and
\begin{align*}
\left|\frac{x}{1+\pi(x)} - (\log(x)-1)\right|
  &\leq \frac{1}{2}x  + \left|\log(x)-1\right|
\le 11,
\end{align*}
for $x\leq 17$, we obtain
\begin{align*}
\left|\frac{x}{1+\pi(x)} - (\log(x)-1)\right|
  &\leq  11.
\end{align*}
Consequently, $K\leq 34.1$ and thus, by \eqref{eq:R2boundprevf},
\begin{align}\label{eq:R2arratiaboundf}
R_{2}
  &\leq \frac{55.6\log\log(n)}{\log(n)}.
\end{align}
The result follows from \eqref{eq:Arratiabound}, \eqref{eq:R1arratiaboundf} and \eqref{eq:R2arratiaboundf}.
\end{proof}
The following lemma will be useful when studying the Poisson approximations for  $\nphi(J_n)$.
\begin{Lemma}\label{lem:arra2}
Let $H_n$ and $\{Q(k)\}_{k\geq 1}$ be as before. Then, for $n\ge 21$,
\begin{align*}
 \PP[Q(n/H_n)\mbox{ divides } H_n]\leq \frac{6.4\log\log(n)}{\log(n)}.
\end{align*}
\end{Lemma}
\begin{proof}
Define $T_{n}:=\PP[Q(n/H_n)\mbox{ divides } H_n]$, and let $\mathcal{P}_{m}^{*}$ be given as in \eqref{eq:Pckstardef}.
Recall that $\dtv(X,X')\le \PP[X\neq X']$ for any coupling of $(X,X')$ defined on $(\Omega,\Fc,\Pb)$. As a consequence,
\begin{align*}
&\dtv(\omega(H_nQ(n/H_n)),  \omega(H_n)+1)\le T_{n}\\
  &\le \PP[Q(n/H_n)\mbox{ divides } H_n,H_n\leq n/2]+\PP[H_{n}\geq n/2].
\end{align*}
The second term is bounded by $\frac{1.5}{\log(n)}$ due to \eqref{ineq:Hntailestimate} and the condition $n\geq 21$. Consequently,
\begin{align*}
T_{n}
  &\le \frac{1.5}{\log(n)}  +  \frac{1}{L_n} \sum_{k=1}^{n/2}   \frac{1}{k(1+ \pi(\lfloor n/k\rfloor))}\sum_{p\in\Pc_{\lfloor n/k\rfloor}^*}\Indi{p|k}\\
  &\le \frac{1.5}{\log(n)}  +\frac{1}{L_n} \sum_{k=1}^{n/2}   \frac{1}{k(1+ \pi(\lfloor n/k\rfloor))}\big(1+\sum_{p\in\Pc_{k} }\Indi{p|k}\big).
\end{align*}

Thanks to \eqref{eq:pinlowerb}, for all $x\ge 2$, we have $\pi(x)\ge 0.67 x/\log(x)$. Hence,
\begin{align*}
T_{n}
  &\le \frac{1.5}{\log(n)}  +\frac{1.5}{ n L_n} \sum_{k=1}^{n} \sum_{p\in\Pc_k} \Indi{p|k} \log(n/k)
	\le \frac{1.5}{\log(n)}  +\frac{1.5}{L_n}\E[\omega(J_n)\log(n/J_n)].
\end{align*}
By an integral comparison, we can easily show that
\begin{align*}
\E[\log(n/J_n)^2]
  &\leq\frac{1}{n}\sum_{m=1}^n\log(n/m)^2
	 \leq\int_0^{1}\log(1/x)^2dx=2.
\end{align*}
Combining this inequality with \eqref{eq:Eomegasqu} and Cauchy-Schwarz inequality, we thus get
\begin{align*}
T_{n}
  &\le \frac{1.5}{\log(n)}  +\frac{5\log\log(n)}{L_n}\leq \frac{6.4\log\log(n)}{\log(n)},
\end{align*}
where in the last inequality we used the condition $n\geq 21$.
\end{proof}

\subsection{Stein's method for normal approximation}
The so-called ``Stein's method'' is a collection of probabilistic techniques that allow to asses the distance between to probability distributions by means of differential operators. It was first introduced in the pathbreaking paper \cite{Stein} by Charles Stein, for obtaining Gaussian approximations. 

The basic idea for Stein's method for Gaussian approximations consists on noticing that if $\Gauss$ is a random variable with standard Gaussian distribution and $f:\R\rightarrow\R$ is an absolutely continuous function satisfying $\E[|f^{\prime}(\Gauss)|]<\infty$, then $\E[\mathcal{A}[f](\Gauss)]=0$, where $\St$ is the so called ``Stein's characterizing operator'', which is defined over the set of differentiable functions, and maps $f$ to $\St[f]$, where $\mathcal{A}[f](x):=f^{\prime}(x)-xf(x)$. Then, at a heuristic level, if $F$ is a random variable with the property that $\E[\St[f](F)]$ is close to zero for a large class of absolutely continuous functions $f$, then $F$ has be close (in some meaningful probabilistic sense), to $\Gauss$.\\

\noindent This heuristics can be formalized quite beautifully by considering a test function $h:\R\rightarrow\R$, and solving for $f$, in Stein's equation
\begin{align}\label{eq:steineq}
\St[f](x)
  &=h(x)-\E[h(\Gauss)].
\end{align}
This way, if $f_h$ denotes  the solution of \eqref{eq:steineq}, then for every family of functions $\mathcal{K}$ satisfying the property that \eqref{eq:steineq} has a  solution $f_h$ for all $h\in\mathcal{K}$, then
\begin{align}\label{eq:Steinbound}
d_{\mathcal{K}}(F,\Gauss)
  &:=\sup_{h\in\mathcal{K}}|\E[h(F)]-\E[h(\Gauss)]|
  =\sup_{h\in\mathcal{K}}|\EE \St[f_h](F)|.
\end{align}
Naturally, in order to find sharp bounds for the right hand side of \eqref{eq:Steinbound}, we need knowledge on the regularity properties of the solutions $f_{h}$, for $h\in\mathcal{K}$. The next Lemmas provide some of these properties for the case where $\Kc=\{I(-\infty,z)\ ;\ z\in\R\}$ and the case where $\Kc$ is the class of Lipschitz functions with Lipschitz constant at most $1$.

\begin{Lemma}{\cite[Lemma 2.3]{ChGoSh}}\label{lem:steinsol}
If $h_z=\Indi{-\infty,z}$ for some $z\in\R$, then \eqref{eq:steineq} has a solution $f_{z}$ satisfying
\begin{align*}
\sup_{x\in\R}|f_z(x)|
  \leq\frac{\sqrt{2\pi}}{4}\ \ \ \ \ \text{ and }\ \ \ \ \ \sup_{x\in\R}|f_z^{\prime}(x)|\leq 1
\end{align*}
and for any $u,v,w\in \RR$,
\begin{align}\label{eq:diffsteinsolbound}
|(w+u) f_z(w+u) - (w+v)f_z(w+v)|
  &\le (|u|+|v|)(|w|+\frac{\sqrt{2\pi}}{4}).
\end{align}
Moreover, $x\mapsto xf_z(x)$ is non-decreasing  and $|xf_z(x)|\le 1$ for all $z\in\RR$
\end{Lemma}

\begin{Lemma}{\cite[Lemma 2.4]{ChGoSh}}\label{lem:steinsol_W}
For each $h$ Lipschitz continuous with Lipschitz constant at most $1$, the equation \eqref{eq:steineq} has a solution $f_h$ satisfying
\begin{align*}
\sup_{x\in\RR} |f_h(x)|\le 2, \ \ \ \  \sup_{x\in\RR}|f'_h(x)| \le \sqrt{2/\pi} \ \ \ \ \mbox{ and } \ \ \ \ \sup_{x\in\RR} |f''_h(x)|\le 2.
\end{align*}
\end{Lemma}
\subsection{Stein's method for Poisson approximation}\label{sec:SteinPoisson}
The aforementioned ideas can as well be applied in the context where the target distribution is a Poisson random variable. The first work in this direction is the paper \cite{Chen}
by Chen, where the methodology was introduced and applied in the context of sums of independent, non-necessarily identically distributed
Bernoulli random variables.  A classic reference on Poisson approximation by Stein's method is the book \cite{BaHolst} by Barbour et al.
(the reader is as well referred to the more recent references \cite{ChatSoDi} and \cite{Erha}).
For approximations towards a Poisson random variable $M$
with parameter $\la$, the corresponding Stein operator becomes $\Stp_{\la}[f](x):=\la f(x+1)-xf(x)$,
and the associated Stein's equation is
\begin{align}\label{eq:steinpoissoneq}
\Stp_{\la}[f](x)
  &=h(x)-\E[h(M)].
\end{align}
The idea for obtaining bounds for $\dtv(X,M)$, where $X$ is a random variable supported in the non-negative integers, consists in considering
a test function $h(k)=\Indi{k\in B}$, and estimating $|\E[h(X)-h(M)]|$ with bounds for $|\E[\Stp_{\la}[X]]|$,
where $f$ is the solution of \eqref{eq:steinpoissoneq}. This task is achieved by means of the following result.
\begin{Lemma}{\cite[page]{BaHolst}}\label{lem:steinsolPo}
If $h(k)=\Indi{k\in B}$ for some $B\subset\NN_0$, then \eqref{eq:steinpoissoneq} has a  solution $f_{h}$ satisfying
\begin{align*}
\sup_{x\in\N_0}|f_h(x)|
  \leq 1\wedge\lambda^{-\frac{1}{2}}\ \ \ \ \ \text{ and }\ \ \ \ \ \sup_{x\in\N_0}|f_h(x+1)-f_h(x)|\leq (1-e^{-\lambda})\lambda^{-1}.
\end{align*}
\end{Lemma}
\subsection{Conditioned independence  of prime factorizations}
In this section we present the key ingredient for our approach, which is a result exhibiting a conditioned independence structure for
the factors of the prime factorization of the random variable $H_n$.
In Appendix \ref{sec:generalKeystep}, we will show that this type of phenomenology extends to a much more general family of probability
distributions (see Remark \ref{Remgeneralkey}).\\

\noindent Our starting point is the well-known relation of the $p$-adic valuation of $J_n$, given by $\alpha_p(J_{n})$, and geometric random variables.
Let $\{\xi_{p}\}_{p\in\Pc}$ be a family of independent geometric random variables with
\begin{align*}
\Pb[\xi_p=k]
  &=p^{-k}(1-p^{-1}).
\end{align*}
Then, for any $i\in \NN$ and $k_1,...,k_i\in\NN_0=\NN\cup\{0\}$ one has
\begin{align*}
\PP[\al_{p_1}(J_n)\ge k_1, \cdots, \al_{p_i}(J_n)\ge k_i ] \to \PP[\xi_{p_1}\ge k_1, \cdots, \xi_{p_i}\ge k_i].
\end{align*}
as $n\to \infty$, where $p_1,...,p_i$ are the first $i$ primes. To see this, one simply notices that
\begin{align*}
\bigcap_{j=1}^i\{\al_{p_j}(J_n) \ge k_j \} = \bigcap_{j=1}^i\{ p_j^{k_j} \mbox{ divides } J_n \} =\left\lbrace \prod_{j=1}^i p_j^{k_j} \mbox{ divides } J_n \right\rbrace,
\end{align*}
 then apply \eqref{eq:probdivide}. This hinges on the intuition that different $p$-adic valuations at $J_n$ become more
independent as $n$ grows to infinity.
Although the the asymptotic independence is our main guiding principle for drawing interesting probabilistic conclusions,
it is not directly applicable to obtain non-asymptotic bounds.  One of the new probabilistic inputs of this paper is the following
non-asymptotic conditioned independence structure of $\al_p(H_n)$.

\begin{Theorem}\label{thm:multiplicities}
Suppose that $n\geq 21$, and let $\{\xi_{p}\}_{p\in\Pc}$ and $L_n$ be given as before. Define the event
\begin{align}\label{eq:Andef}
\scr{A}_n
  &:=\Big\{\prod_{p\in\mathcal{P}_n}p^{\xi_{p}}\leq n\Big\},
\end{align}
as well as the random vector $\vec{C}(n):=(\alpha_{p}(H_n) ; p\in\Pc_n)$. Then
\begin{align}\label{eq:PrAnineq}
\PP[\scr{A}_n]
  &=L_n\prod_{p\le n}(1-p^{-1})\geq\frac{1}{2},
\end{align}	
and
\begin{align}\label{eq:conditionallaws}
    \Lc(\vec{C}(n))
      &=\Lc(\vec{\xi}(n)\ |\ \scr{A}_n),
\end{align}
where $\vec{\xi}(n):=(\xi_p ; p\in\Pc_n)$. In particular,
\begin{align}\label{eq:conditionallaw20}
\mathcal{L}(\nphi(H(n)))= \mathcal{L}(\sum_{p\in\mathcal{P}_n} \psi(p^{\xi_p})| \scr{A}_n).
\end{align}
\end{Theorem}
\begin{Remark}\label{Remgeneralkey}
As one can observe from the proof that we present below, the heuristic explanation of why the above result holds,
comes from the fact that the probability mass function of the geometric distribution transforms ``products over sets of primes'' to
``sums over sets of integers'', which allows us to use the multiplicative property of characteristic functions in our advantage. One thus can naturally ask whether Theorem \ref{thm:multiplicities} can be extended to more general families of distributions.
This task can indeed be carried without difficulties, provided that we impose some multiplicativity condition on the underlying
random variable, as we explain in Appendix \ref{sec:generalKeystep}. For the purposes of this manuscript, we mostly require knowledge on $\mathcal{L}(H_n)$, so
in this section we only handle the case of the Harmonic distribution and leave its generalized version Proposition \ref{thm:multiplicities2}  as
as an available tool for future related problems.
\end{Remark}
\begin{proof}[Proof of Theorem \ref{thm:multiplicities}]
Consider a fixed vector $\vec{\lambda}=(\lambda_p ; p\in\Pc_n)\in\R^{\pi(n)}$. For a given $n\in\N$, define the set
\begin{align}\label{eq:mathKndef}
\mathcal{K}_{n}
  &:=\{(c_p;\ p\in\Pc_n)\in\N_0^{\pi(n)}\ ;\ \prod_{p\in\Pc_n}p^{c_p}\leq n\},
\end{align}
consisting of the tuples of non-negative integers $c_p,$ indexed by the primes $p$ belonging to $\Pc_n$ and satisfying the condition $\prod_{p\in\Pc_n}p^{c_p}\leq n$.

Observe that the prime factorization theorem induces a natural bijective correspondence between the sets $\mathcal{K}_{n}$ and $\{1,\dots, n\}$.
Let $f:\NN_0^{\pi(n)}\to \RR$ be bounded. The bijection allows us to write
\begin{align*}
\EE[f(\vec\xi(n))\1(\scr A_n)] &= \sum_{\vec c=(c_p;\ p\in\Pc_n)\in\mathcal{K}_n} f(\vec{c}) \  \Pb[\xi_p=c_p \text{ for all }p\in\Pc_n] \\
&=  \sum_{\vec{c}=(c_p;\ p\in\Pc_n)\in\mathcal{K}_n} f(\vec{c}) \prod_{p\in\Pc_n} (1-p^{-1}) \prod_{p\in\Pc_n} p^{-c_p}\\
&= \sum_{k=1}^n f(\al_p(k),p\in\Pc_n)  k^{-1} \prod_{p\in\Pc_n} (1-p^{-1}) \\
&= \EE[f(\al_p(H_n), p\in\Pc_n)] L_n \prod_{p\in\Pc_n} (1-p^{-1}).
\end{align*}
Letting $f(\vec{c})=1$ for all $\vec{c}\in \NN_0^{\pi(n)}$, we deduce that
\begin{equation}
\label{eq:probA_n}
\PP[\scr{A}_n]=L_n\prod_{p\le n}(1-p^{-1}),
\end{equation}
which in addition gives
\begin{align*}
    \E[f(\vec{\xi}(n)) | \scr{A}_n]
      &=\E[f(\alpha_p(H_n),p\in\Pc_n)],
\end{align*}
hence implying \eqref{eq:conditionallaws}.\\

\noindent To prove the inequality \eqref{eq:PrAnineq}, we use the identities \eqref{eq:probA_n} and \eqref{bound:primeprod}, as well as the fact that $L_{n}\geq\log(n+1)$ and $n\geq 21$, in order to obtain
\begin{align*}
\PP[\scr{A}_n]
  &	\geq e^{-\gamma}(1-\log(x)^{-2})
	\geq 0.5.
\end{align*}	

\noindent Finally, we notice that identity \eqref{eq:conditionallaw20} easily follows from \eqref{eq:conditionallaws}, since
\begin{align*}
    \mathcal{L}(\nphi(H_n))
    &=\mathcal{L}(\sum_{p\in\Pc_n}\nphi(p^{\alpha_p(H_n)}))
    =\mathcal{L}(\sum_{p\in\Pc_n}\nphi(p^{\xi_p}) | \scr{A}_n).
\end{align*}
\end{proof}

\subsection{Poisson embedding and linear approximation}\label{sec:embedding}

Another key idea for proving Theorem \ref{thm:CLTHarmonic} consists on regarding the law of $\psi(H_n)$ as the distribution
of a suitable functional of a Poisson point process. In view of Theorem \ref{thm:multiplicities}, this task can be carried simply by viewing
the $\xi_p$'s as functionals of a Poisson point process.
To achieve this, we define the (discrete) ambient space $\mathbb X := \{(p,k): p\in\mathcal P, k\in\NN_0 \}$
and consider a Poisson process $\eta$, defined on $\mathbb X$, with intensity measure $\la:\mathbb{X}\rightarrow\R_{+}$ given by
\begin{align*}
\la(p,k) = \frac{1}{kp^k}, \quad \mbox{ for all } p\in\mathcal P, k\in\NN.
\end{align*}
Using an elementary manipulation of characteristic functions, one can easily show that if $\xi$ is a   geometric random variable
satisfying $\PP[\xi= k]=(1-\rho)\rho^k$ for $k\in\NN_0$ and $\rho\in(0,1)$, then
\begin{align*}
\xi \overset{Law}=  \sum_{k\ge 1} k M_\rho(k)
\end{align*}
where the random variables $M_\rho(k)$ indexed by $\NN$ are independent Poisson with parameters $\frac{\rho^k}{k}$, respectively.
Applying this result to the variables $\xi_{p},$ we obtain the identity in law
\begin{align}\label{eq:xipstochint}
(\xi_p, p\in\mathcal P)
  \stackrel{Law}{=}\Big(\sum_{k\in\NN} k \eta(p,k), p\in\mathcal P\Big).
\end{align}
Taking  \eqref{eq:xipstochint} into consideration, we will assume in the sequel that
\begin{align}\label{eq:xipfuncrep}
\xi_p
  &:=\sum_{k\in\NN} k \eta(p,k)
\end{align}
for every $p\in\Pc$. Observe that the additivity of $\nphi$ implies that
\begin{align}\label{eq:sumppsi}
\nphi(\prod_{p\in\mathcal P_n} p^{\xi_p}) = \sum_{p\in\mathcal P_n} \nphi(p^{\xi_p}),
\end{align}
which induces a natural dependence of the law of $\psi(H_n)$ on the underlying Poisson process $\eta$ via Theorem \ref{thm:multiplicities}. However, for
computational simplicity, we will instead use the identity $\mathbbm{1}(\xi_p=1) = \xi_p - \xi_p \mathbbm{1}(\xi_p\ge 2)$ to write
\begin{align}\label{eq:YnRndecomp}
\nphi(\prod_{p\in\mathcal P_n} p^{\xi_p})
  &
	= Y_n+R_n,
\end{align}
where
\begin{align}\label{eq:YnRndecompq}
Y_n
  :=\sum_{p\in\mathcal P_n}\nphi(p)\xi_p\ \ \ \ \ \ \ \ \text{ and }\ \ \ \ \ \ \ \
R_n
  :=\sum_{p\in\mathcal P_n}(\nphi(p^{\xi_p})-\nphi(p) \xi_p) \mathbbm{1}(\xi_p\ge 2),
\end{align}
The decomposition \eqref{eq:YnRndecomp} will be of great help for future computations, due to the fact that
$Y_n$ has compound Poisson distribution, and $R_n$ is an error satisfying
\begin{align}\label{eq:Rnmainbound}
\E[|R_n|]
  &\leq c_1 + 2c_2,
\end{align}
due to Lemma \ref{l:reduction2} in the Appendix.

\subsection{Integration by parts for linear functionals of $\eta$}
As it is usually the case in Stein's method, a suitable integration by parts
greatly simplifies computations.
This task can be addressed by using the Poisson integral structure of $Y_n$ and
the following integration by parts formula, obtained as an easy consequence of Mecke's formula (or from palm theory for Poisson processes). In what follows, we write $\mu(\rho):= \int_{\mathbb X} \rho(x)\mu(dx)$ for any positive measure $\mu$ on $\mathbb X$ and function $\rho:\mathbb X\to \RR$.

\begin{Lemma}\label{lem:tildegIBP}
Write $\wt\eta(\rho)=\eta(\rho) - \la(\rho)$ for the compensated Poisson integral of a kernel function
$\rho\in L^1(\mathbb{X},d\la)\cap L^2(\mathbb{X},d\la)$. Let $G:=G(\eta)$ be square-integrable $\sigma(\eta)$-measurable.  Then
\begin{align}\label{eq:tildegIBP}
\EE[ \wt\eta(\rho) G(\eta) ] = \int_{\X} \rho(x) \EE[D_x G(\eta) ] \la(dx),
\end{align}
where $D_x G(\eta):= G(\eta+\de_x) -G(\eta)$.
\end{Lemma}
\begin{proof}
Mecke's equation \cite[Theorem 4.1]{LP}  states that for $\rho$ and $G$ as above,
\begin{align*}
\EE[ \eta(\rho) G(\eta) ]  =  \EE\Big[ \int_{\X}  \rho(x) G(\eta) \eta(dx)\Big] = \int_{\X} \EE[ \rho(x) G(\eta+\de_x) ] \la(dx).
\end{align*}
Subtracting from both sides $\EE[\la(\rho) G(\eta)]$, one arrives at
\begin{align*}
\EE[\wt \eta(\rho) G(\eta) ]  = \int_{\X} \EE[ \rho(x)(G(\eta+\de_x) -G(\eta))]\la(dx) = \int_{\X} \rho(x) \EE[D_x G(\eta) ] \la(dx),
\end{align*}
as required.
\end{proof}

\begin{Remark}
\
\begin{enumerate}
\item[(i)]
The lemma above is in fact a duality formula: the compensated Poisson measure applied to a
deterministic function is the dual operation of the difference operator $D$, which is customarily called the Kabanov-Skorohod integral. We stress that duality (or  integration by parts) formula on the Poisson space holds for more general Poisson functionals, and reduces to our case when applied to linear ones.
We refer the interested reader to the monographs \cite{NP,LP}  for more identities of this kind, and to Last, Peccati and Schulte \cite{LPS16}, D\"obler and Peccati \cite{DP18} and a recent work \cite{LrPY20+} for more striking applications of duality formulas for normal approximation on the Poisson space.
\item[(ii)] The integrability condition is satisfied automatically for each $\rho_n$ and $\varrho_n$ utilized in the proofs as they are supported on subsets of $\mathbb X$ with finite $\la$-measure.
\end{enumerate}
\end{Remark}

\section{Proof of Theorem \ref{thm:CLTHarmonic}: Wasserstein bound}\label{mainthmWboundCLT}

\noindent This section is devoted to proving equation \eqref{eq:Wasserstein}.
Recall that by Theorem \ref{thm:multiplicities}, we have $\Lc(\psi(H_n)) = \Lc(Y_n+R_n|\mathscr{A}_n)$, where $Y_n$ and $R_n$ are given by
\eqref{eq:YnRndecompq} and $\scr{A}_n$ by \eqref{eq:Andef}. Define $W_n$, $\tilde{W}_n$ by
\begin{align}\label{eq:WandWtildedef}
W_n
  :=\sigma_n^{-1}(Y_n - \mu_n )
\ \ \ \ \ \ \ \ \text{ and }\ \ \ \ \ \ \ \
\tilde{W}_n
  :=\sigma_n^{-1}(\sum_{p\in\Pc_n}\nphi(p^{\xi_p}) - \mu_n ),
\end{align}
so that  the law of the underlying approximating sequence
\begin{align}\label{eq:Zndef}
Z_n=Z_n^\psi
  &:=\sigma_n^{-1}(\psi(H_n) - \mu_n),
\end{align}
can be written as
\begin{align}\label{eq:Wtildelawdecomp}
\mathcal{L}(Z_n)
  &= \mathcal{L}(\tilde{W}_n | \scr{A}_n)
	=\mathcal{L}(W_n + \sigma_n^{-1}R_n | \scr{A}_n).
\end{align}
 Let $I_n$ denote the indicator  of $\scr{A}_n$.
From \eqref{eq:Wtildelawdecomp}, \eqref{eq:PrAnineq} and the definition of $1$-Wasserstein distance, it follows that
\begin{align*}
d_{1}(Z_n,\mathcal{L}(W_n | \scr{A}_n)) \le \sigma_n^{-1}\EE[|R_n| | \scr{A}_n] \le 2 \sigma_n^{-1}\EE[|R_n| I_n],
\end{align*}
which by the triangle inequality and \eqref{eq:Rnmainbound}   implies that
\begin{align}\label{eq:conditionedsteinprev}
d_{1}(Z_n,N)
  &\leq d_{1}(\mathcal{L}(W_n | \scr A_{n}),N)+\sigma_n^{-1}( 4c_1 + 2 c_2).
\end{align}
We thus have reduced the problem to bounding $d_{1}(\mathcal{L}(W_n | \scr A_{n}), N)$.
This task can be achieved by implementing Stein's bound \eqref{eq:Steinbound} for the conditional law $\mathcal{L}(W_n | \scr A_{n})$ and then
bounding from above the quantity
\begin{align}\label{eq:conditionedstein}
|\E[\St[f]( W_n) | \scr A_{n}]|
  &= \Pb[\scr{A}_n]^{-1}|  \EE[ f(W_n)W_n I_n ] - \EE[f'(W_n)I_n ]|\nonumber\\
	&\leq 2|  \EE[ f(W_n)W_n I_n ] - \EE[f'(W_n)I_n ]|,
\end{align}
where $f$ is the solution to Stein's equation \eqref{eq:steineq} with respect to a test
function $h\in$ Lip$(1)$ satisfying the uniform bounds in Lemma \ref{lem:steinsol_W}. \\

\noindent\textbf{Step I}\\
\noindent In order to bound \eqref{eq:conditionedstein}, we next find a suitable integration by parts for $W_n$.
Recall from Section \ref{sec:embedding} that $\xi_p=\xi_p(\eta)=\sum_{k\ge 1} k \eta(p,k)$. Thus, we can write
\begin{align*}
W_n
  &=  \frac{1}{\sigma_n} \sum_{p\in\mathcal P_n}\sum_{k\ge 1} \nphi(p)  k \Big( \eta(p,k) - \frac{1}{kp^k}\Big)=\wt\eta(\rho_n),
\end{align*}
where
\begin{align}\label{eq:rhondef}
\rho_n(k,p)
  &:=\sigma_n^{-1}k\nphi(p) \Indi{p\in\mathcal P_n}.
\end{align}
From the definition of the difference operator $D_x$, one can easily check that if $F_1$ and $F_2$ are $\sigma(\eta)$-measurable, then
\begin{align}\label{eq:DFG}
D_x (F_1F_2) &= F_1 D_x F_2 + F_2 D_x F_1 + D_x F_1 D_x F_2.
\end{align}
Thus, applying Lemma \ref{lem:tildegIBP} to $G = f(W_n) I_n$, we can write
\begin{align}\label{eq:etarhodecompMrr}
 \EE[ f(W_n)W_n I_n ]
  &=\EE[ \wt\eta(\rho_n) f(\wt\eta(\rho_n)) I_n]
  = J_n + \varepsilon_n,
\end{align}
where
\begin{align*}
J_n
  &:= \int_{\X} \rho_n(x)  \EE[ I_n D_x(f(\wt\eta(\rho_n)))] \lambda(dx) \\
\varepsilon_n
  &:= \int_{\X} \rho_n(x)  \EE[ (f(\wt\eta(\rho_n)) + D_x f(\wt\eta(\rho_n)))D_x I_n] \lambda(dx).
\end{align*}
Next we describe separately the behavior of each term on the right hand side of \eqref{eq:etarhodecompMrr}.\\

\noindent \textbf{Step II}\\
First we analyze $J_n$. To this end, we use Taylor's formula  to write
\begin{align*}
|D_x(f(\wt\eta(\rho_n))) - f'(\wt\eta(\rho_n)) \rho_n(x)|
  &= |f(\wt\eta(\rho_n)+\rho_n(x)) - f(\wt\eta(\rho_n)) -f'(\wt\eta(\rho_n)) \rho_n(x)|\\
  &\leq \|f^{\prime\prime}\|_{\infty}|\rho_n(x)|^2
	\leq \sqrt{2/\pi}|\rho_n(x)|^2,
\end{align*}
where the last inequality follows from Lemma \ref{lem:steinsol_W}. It is clear that $W_n$ is standardized by our choice of
$\mu_n$ and $\sigma_n$, and  consequently
 $\|\rho_n\|_{L^2(\mathbb{X},d\la)}=\Var[W_n]=1$ by moment formula for Poisson integrals \cite[Lemma 12.2]{LP}, yielding
\begin{align*}
|J_n-\EE[ I_n f'(W_n)]|
  &=\Big|J_n-\int_{\X} \rho_n(x)^2\EE[ I_n f'(\wt\eta(\rho_n))] \la(dx) \Big|
  \leq \sqrt{2/\pi}\int_{\X} |\rho_n(x)|^3 \la(dx).
\end{align*}
From  \eqref{eq:rhondef} and Lemma \ref{eq:Lemageometricaux}, we have
\begin{align*}
\int_{\X} |\rho_n(x)|^3 \la(dx)
  &=\frac{1}{\sigma_n^3} \sum_{p\in\mathcal P_n}\sum_{k\ge 1} \frac{|\nphi(p)|^3k^2}{p^k}
	\leq\frac{c_1}{\sigma_n^{3}}\sum_{p\in\Pc_n}\frac{|\nphi(p)|^2(1+p^{-1})}{p(1-p^{-1})^3}\\
	&=\frac{c_1}{\sigma_n^3} \sum_{p\in\mathcal P_n} \frac{|\nphi(p)|^2\text{Var}[\xi_p](1+p^{-1})}{1-p^{-1}}
	\leq 3c_1\frac{1}{\sigma_n^3} \sum_{p\in\mathcal P_n} \text{Var}[\nphi(p)\xi_p]
	=\frac{3c_1}{\sigma_n}.
\end{align*}
We thus conclude that
\begin{align}\label{eq:boundMfinal}
|J_n-\EE[ I_n f'(W_n)]| \leq \frac{3\sqrt{2/\pi}c_1}{\sigma_n}.
\end{align}

\noindent \textbf{Step III}\\
\noindent Next we handle the term $\varepsilon_n$. We have
\begin{align}\label{eq:ep_n_step3}
|\varepsilon_n|
 &\le 3 \|f\|_{\infty} \int_{\X} |\rho_{n}(x)| \EE[ | D_x I_n |] \la(dx)
 \le 6\int_{\X} |\rho_n(x)| \EE[ | D_x I_n |] \la(dx),
\end{align}
where the last inequality follows from Lemma \ref{lem:steinsol_W}.
Hence, it suffices to consider $\EE[|D_x I_n|]$. We notice that $|D_x I_n|\in \{0,1\}$, more precisely, 
setting $x=(p_0,k_0)\in \Pc_n\times \NN$, we have
\begin{align*}
D_x I_n
  &= \1( p_0^{ k_0}  \prod_{p\in\Pc_n} p^{\sum_{k=1}^\infty k\eta(p,k)}\le n ) - \1(\prod_{p\in\Pc_n} p^{\sum_{k=1}^\infty k\eta(p,k)}\le n)\\
  &= - \1(\prod_{p\in\Pc_n} p^{\xi_p}\in [p_0^{-k_0} n, n] )
\end{align*}
Hence, $D_x I_n\neq 0$ implies $I_n\neq 0$. Applying Theorem \ref{thm:multiplicities} and Lemma \ref{l:HnTail}, we have
\begin{align*}
\EE[ |D_x I_n|]
&=\EE[I_n|D_x I_n|]
=\PP[\scr A_n] \EE[ |D_x I_n| | \scr{A}_n]\\
&\leq \Pb[ \prod_{p\in\Pc_n} p^{\xi_p}\in [p_0^{-k_0} n, n]  | \scr{A}_n]= \PP[ H_n \ge n p_0^{-k_0}]
 \le \frac{2k_0 \log (p_0)}{\log(n)}.
\end{align*}
Plugging this estimate back to \eqref{eq:ep_n_step3} and applying Lemma \ref{eq:Lemageometricaux}, we obtain
\begin{align}\label{eq:epsilonnbound}
|\varepsilon_n|
  &\le  12  \sum_{p\in\mathcal P_n} \sum_{k\ge 1} \frac{1}{\sigma_n} |\nphi(p)|  \frac{\log(p)}{\log(n)} \frac{k}{p^k}\nonumber\\
  &\le  \frac{48c_1}{\sigma_n} \sum_{p\in\mathcal P_n} \frac{\log(p)}{ p\log(n)}
\le  \frac{48c_1}{\sigma_n} ,
\end{align}
where the last inequality follows from Merten's formula \eqref{eq:Mertenslog}.
Inequality \eqref{eq:Wasserstein} follows from \eqref{eq:conditionedsteinprev}, \eqref{eq:conditionedstein},
\eqref{eq:etarhodecompMrr}, \eqref{eq:boundMfinal} and \eqref{eq:epsilonnbound}.\\

\section{Proof of Theorem \ref{thm:CLTHarmonic}: Kolmogorov bound}\label{mainthmKboundCLT}

\noindent Now we proceed with the proof of \eqref{eq:Kolmogorov}.
Let $W_n,\tilde{W}_n$ and $Z_{n}$ be given as in \eqref{eq:WandWtildedef} and \eqref	{eq:Zndef}.
It is not evident that $\mathcal L(\tilde W_n | \scr{A}_n)$ and $\mathcal L(W_n|\scr{A}_n)$ are close in the Kolmogorov distance
as they are in the Wasserstein distance, so have to implement Stein's method directly for $\mathcal L(Z_n) =\mathcal L(\tilde W_n|\scr{A}_n)$. We will see, however, that the decomposition \eqref{eq:Wtildelawdecomp} is still relevant in order to apply our integration by parts formula Lemma  \ref{lem:tildegIBP}.

It suffices to bound $|\E[\St[f](Z_n)]|$, for $f$ given as the solution to the Stein's equation associated to a test function $h$ of the form $h(x):=\Indi{x\leq z}$, for $z\in\R$.
By Theorem \ref{thm:multiplicities}, we have that
\begin{align}\label{e:5.1}
|\E[\St[f](Z_n)]|
  &= \Pb[\scr{A}_n]^{-1}|  \EE[ f(\tilde{W}_n)\tilde{W}_n I_n ] - \EE[f'(\tilde{W}_n)I_n ]|\nonumber\\
	&\leq 2|  \EE[ f(\tilde{W}_n)\tilde{W}_n I_n ] - \EE[f'(\tilde{W}_n)I_n ]| \nonumber\\
	&\le \frac{\sqrt{2\pi}}{2} \sigma_n^{-1} \EE[|R_n|] + 2|  \EE[ f(\tilde{W}_n) W_n I_n ] - \EE[f'(\tilde{W}_n)I_n ]|,
\end{align}
where we have used \eqref{eq:Wtildelawdecomp} and Lemma \ref{lem:steinsol} for the last inequality. Observe that
by relation \eqref{eq:Rnmainbound},
\begin{align}\label{e:5.1p}
\frac{\sqrt{2\pi}}{2}\E[|R_{n}|]
  &\leq1.3 c_1+2.6 c_2,
\end{align}
so it suffices to estimate the second term in the right hand side of \eqref{e:5.1}. As before, we split the rest of the proof into several steps.\\

\noindent\textbf{Step I}\\
In order to handle the second term in \eqref{e:5.1}, we apply Lemma \ref{lem:tildegIBP} and \eqref{eq:DFG} to obtain
\begin{align}\label{e:5.1pp}
\EE[ W_n f(\tilde W_n) I_n]
  &= \int_{\X} \rho_n(x) \EE[ D_x(f(\tilde W_n)I_n)] \la(dx) = T_n + \ep_{n,1},
\end{align}
where
\begin{align}
T_n &:= \int_{\X} \rho_n(x) \EE[ D_x(f(\tilde W_n)) I_n] \la(dx), \label{eq:Tndefsom}\\
\ep_{n,1}&:=  \int_{\X} \rho_n(x) \EE[ (f(\tilde W_n)+ D_x(f(\tilde W_n)) D_x I_n)] \la(dx)\nonumber.
\end{align}
Notice that Lemma \ref{lem:steinsol}, together with the argument leading to \eqref{eq:epsilonnbound}, yields
\begin{align}\label{e:ep_n1final}
|\ep_{n,1}| \le 3 \|f\|_\infty \int_{\X} \rho_n(x) \EE[|D_x I_n|]\la(dx) \le     \frac{15.1c_1}{\sigma_n},
\end{align}
so that we are  left to show that $T_n$ and $\EE[f'(\tilde W_n)I_n]$ are close.\\

\noindent\textbf{Step II}\\
In view of the form of $T_n$, it helps to rewrite $\EE[f'(\tilde W_n)I_n]$ as follows
\begin{align*}
\EE[f'(\tilde W_n)I_n] = \int_{\X} \rho_n(x)  \EE[\rho_n(x) f'(\tilde W_n)I_n] \la(dx),
\end{align*}
where we used the fact that $\int_{\X} \rho_n^2 d\la=1$.
  On the other hand, by Taylor's formula,
\begin{align*}
D_x(f(\tilde W_n)) - f'(\tilde W_n) D_x \tilde W_n = D_x \tilde W_n  \int_0^1 (f'(\tilde W_n + tD_x \tilde W_n) - f'(\tilde W_n)) dt.
\end{align*}
Therefore, approximating $D_x(f(\tilde W_n))$ by $D_x\tilde W_n f'(\tilde W_n)$ in \eqref{eq:Tndefsom}, and
then $D_x \tilde W_n$ by $\rho_n(x)$, we have
\begin{align}\label{eq:Tntotarget}
T_n - \EE[f'(\tilde W_n)I_n] = \ep_{n,2} + \ep_{n,3},
\end{align}
where
\begin{align*}
\ep_{n,2}&:= \int_0^1 \int_{\X} \rho_n(x) \EE[ D_x \tilde W_n (f'(\tilde W_n + tD_x \tilde W_n) - f'(\tilde W_n))  I_n ]  \la(dx) dt, \\
\ep_{n,3}&:=  \int_{\X} \rho_n(x) \EE[ (D_x \tilde W_n - \rho_n(x)) f'(\tilde W_n) I_n ] \la(dx).
\end{align*}
To bound both terms, we have to understand $D_x \tilde W_n$. Since $\tilde W_n$ is a non-linear functional of $\eta$, the quantity $D_x \tilde W_n$ would be random, in contrast to $D_x W_n$.  Write $x=(p,k)\in \Pc_n\times \NN$, we have
\begin{align}\label{e:DxtildeW}
D_x\tilde W_n  = \tilde W_n(\eta+\de_x) - \tilde W_n(\eta)= \sigma_n^{-1}(\nphi(p^{\xi_{p}+k}) - \nphi(p^{\xi_p})).
\end{align}
Indeed, the additional Dirac mass $\de_{x}$ that is added to $\eta$ will affect only  one summand in the definition of $\tilde W_n$, more precisely, for all $q\in\Pc_n$,
\begin{align}\label{e:xi_add_one}
\xi_q(\eta+\de_x) = \sum_{j\ge 1} j\  (\eta+\de_x)(q,j) = \begin{cases}
k + \sum_{j\ge 1} j \eta(q,j)=k + \xi_p, & \mbox{ if } p=q, \\
\sum_{j\ge 1} j \eta(p,j)= \xi_p, & \mbox{ if } p\neq q.
\end{cases}
\end{align}

\noindent{\textbf{Step III}}\\
\noindent We first bound the error $\ep_{n,3}$.  By Lemma \ref{lem:steinsol} and the explicit form of $\rho_n$ and $\la$, we have
\begin{align}\label{e:ep_n3}
  |\ep_{n,3}| &\le  \int_{\X} |\rho_n(x)| \EE[| D_x \tilde W_n -\rho_n(x) |] \la(dx)  \\
   &= \frac{1}{\sigma_n^2}\sum_{p\in\Pc_n} \sum_{k\ge 1} \frac{|\psi(p)|}{p^k} \EE[|\nphi(p^{\xi_{p}+k}) - \nphi(p^{\xi_p}) - k\nphi(p)|]  \notag\\
   &\le  \frac{c_1}{\sigma_n^2} \sum_{p\in\Pc_n} \frac{1}{p} \EE[|\nphi(p^{\xi_{p}+1}) - \nphi(p^{\xi_p}) - \nphi(p)|]  \notag\\
  &\quad + \frac{c_1}{\sigma_n^2} \sum_{p\in\Pc_n} \sum_{k\ge 2} \frac{1}{p^k} \EE[|\nphi(p^{\xi_{p}+k}) - \nphi(p^{\xi_p}) - k\nphi(p)|]=: \ep_{n,3,1}+\ep_{n,3,2}.\notag
  \end{align}
  Notice that the term inside the expectation of $\ep_{n,3,1}$ vanishes under the event $\xi_p =0$. Thus, $\ep_{n,3,1} \le c_1 \sigma_n^{-2} (\de_1 + \de_2 + \de_3)$ with
\begin{align}\label{eq:delta's}
\de_1 &=  \sum_{p\in\Pc_n} \frac{1}{p} \EE[|\psi(p^{\xi_p+1})|\1(\xi_p\ge 1)],\notag \\
\de_2 &= \sum_{p\in\Pc_n} \frac{1}{p} \EE[|\psi(p^{\xi_p})|\1(\xi_p\ge 1)], \\
\de_3&= \sum_{p\in\Pc_n} \frac{|\psi(p)|}{p} \PP[\xi_p\ge 1]. \notag
\end{align}
 By the memoryless property  $\mathcal L(\xi_p | \xi_p\ge k) = \mathcal L(\xi_p + k)$ for all $p\in\Pc$ and $k\in\NN$, as well as the Cauchy-Schwarz inequality,  we have
\begin{align}\label{eq:delta1finalb}
\de_1 = \sum_{p\in\Pc_n} \frac{\EE[|\psi(p^{\xi_p+2})|]}{p^2} \le \zetab(2)^{1/2} c_2.
\end{align}
The same argument leads to
\begin{align}\label{eq:delta2finalb}
\de_2
  &=   \sum_{p\in\Pc_n} \frac{\EE[|\psi(p^{\xi_p+1})|]}{p^2}
	\le \zetab(2)c_1 + \sum_{p\in\Pc_n} \frac{\EE[|\psi(p^{\xi_p+1})|\1(\xi_p\ge 1)]}{p^2}\nonumber\\
	&\le \zetab(2)c_1 + \sum_{p\in\Pc_n} \frac{\EE[|\psi(p^{\xi_p+2})]}{p^3}
	\leq\zetab(2) c_1 + \zetab(4)^{1/2} c_2,
 \end{align}
and $\de_3\le \zetab(2) c_1$, yielding
\begin{align}\label{e:ep_n31}
\ep_{n,3,1}
  &\le  \frac{1}{\sigma_n^2} \big(2\zetab(2) c_1^2+(\zetab(2)^{\frac{1}{2}} + \zetab(4)^{\frac{1}{2}})c_1c_2\big)
	\leq \frac{1}{\sigma_n^2} \big(1.3c_1^2+1.1c_1c_2\big).
\end{align}
To bound $\ep_{n,3,2}$, we write $\ep_{n,3,2}\le c_1\sigma_n^{-2}(\de'_1 + \de_2' +\de'_3)$ where
\begin{align}\label{eq:deltaprime's}
\de'_1 & = \sum_{p\in\Pc} \sum_{k\ge 2} \frac{1}{p^k} \EE[|\nphi(p^{\xi_{p}+k})|],\notag\\
\de'_2 &= \sum_{p\in\Pc} \sum_{k\ge 2} \frac{1}{p^k} \EE[ |\nphi(p^{\xi_p})|], \\
\de'_3&= \sum_{p\in\Pc} \sum_{k\ge 2} \frac{1}{p^k} k|\nphi(p)|.\notag
\end{align}
Notice that by the identity $\mathcal{L}(\xi_p+k)=\mathcal{L}(\xi_p\ |\ \xi_p\geq k)$,
\begin{align}\label{eq:delta1primebound}
 \de'_1
  &= \sum_{p\in\Pc} \sum_{k\ge 2} \EE[ |\psi(p^{\xi_p})|\1(\xi_p\ge k)] \leq   \sum_{p\in\Pc} \EE[|\psi(p^{\xi_p})| \xi_p \1(\xi_p\ge 2)]\nonumber\\
  &= \sum_{p\in\Pc}\frac{1}{p^2} \EE[|\psi(p^{\xi_p+2})|(\xi_p+2)] \le c_2 \Big(\sum_{p\in\Pc} \frac{\EE[(\xi_p+2)^2]}{p^2}\Big)^{1/2} \le (11 \zetab(2))^{1/2} c_2.
\end{align}
  where we used the Cauchy-Schwarz inequality and Lemma \ref{eq:Lemageometricaux} for the two inequalities.   The same argument implies
\begin{align}\label{eq:delta2primebound}
\de'_2
  &= \sum_{p\in\Pc} \frac{(1-p^{-1})^{-1}}{p^2} \EE[|\psi(p^{\xi_p})|(\Indi{\xi_p=1}+\Indi{\xi_p\geq 2})]\nonumber\\
	&\le  \zetab(3) c_1 + 2\sum_{p\in\Pc} \frac{\EE[|\psi(p^{\xi})|\1(\xi_p\ge 2)]}{p^2}\nonumber\\
  &=  \zetab(3) c_1  + 2 \sum_{p\in\Pc} \frac{\EE[|\psi(p^{\xi_p+2})|]}{p^4}
	\le  \zetab(3) c_1 + 2\zetab(6)^{1/2} c_2,
\end{align}
and $\de'_3\le 6 \zetab(2)c_1 $, yielding
\begin{align}\label{e:ep_n32}
\ep_{n,3,2}
  &\le \frac{1}{\sigma_n^2}[(\zetab(3)+6\zetab(2))c_1^2 + ((11\zetab(2))^{\frac{1}{2}} + 2\zetab(6)^{\frac{1}{2}})  c_1c_2 ]
	\leq\frac{1}{\sigma_n^2}(4.3c_1^2 + 3c_1c_2).
\end{align}
Combining \eqref{e:ep_n31} and \eqref{e:ep_n32}, we obtain
\begin{align}\label{e:ep_n3final}
|\ep_{n,3}|
  \le\frac{1}{\sigma_n^2}(5.6c_1^2+4.1c_1c_2).
\end{align}

\noindent{\textbf{Step IV}}\\
It remains to bound $\ep_{n,2}$.  Applying Stein's equation $f^{\prime}(x)=xf(x)+\Indi{x\le z} - \PP[N\le z]$ leads to
\begin{align*}
\ep_{n,2} &= \int_0^1 \int \rho_n(x) \EE[ D_x \tilde W_n ( (\tilde W_n + tD_x \tilde W_n) f(\tilde W_n + tD_x \tilde W_n) - \tilde W_n f(\tilde W_n))  I_n ]  \la(dx) dt\\
&+ \int_0^1 \int \rho_n(x) \EE[ D_x \tilde W_n (\Indi{\tilde W_n + tD_x \tilde W_n\le z} -  \Indi{W_n\le z})  I_n ]  \la(dx) dt
\end{align*}
Since $x\mapsto xf(x)$ and $x\mapsto -\Indi{x\le z}$ are non-decreasing  functions, one has for $h\in\RR, t\in(0,1)$ that
\begin{align*}
h\cdot((x+h)f(x+h) - xf(x))\ge h\cdot((x+th)f(x+th) - xf(x)) \ge 0,
\end{align*}
and
\begin{align*}
h\cdot  (\Indi{x\le z}-\Indi{x+h\le z})\ge  h\cdot  (\Indi{x\le z}-\Indi{x+th\le z}) \ge 0.
\end{align*}
Applying these inequalities with $h=D_x \tilde W_n$ and $x=\tilde W_n$, we obtain
\begin{align*}
|\ep_{n,2}|
  &\le \int_{\X} |\rho_n(x)| \EE[ D_x \tilde W_n ( (W_n + D_x \tilde W_n) f(\tilde W_n + D_x \tilde W_n) - \tilde W_n f(\tilde W_n))] \la(dx) \\
  & + \int_{\X} |\rho_n(x)| \EE[ D_x \tilde W_n (\Indi{W_n\le z}-\Indi{W_n + D_x \tilde W_n\le z} )] \la(dx)
\end{align*}
Observe that
\begin{align*}
(\tilde W_n + D_x \tilde W_n) f(\tilde W_n + D_x \tilde W_n) - \tilde W_n f(\tilde W_n) &= D_x( \tilde Wf(\tilde W)), \\
\Indi{\tilde W_n\le z}-\Indi{\tilde W_n + D_x \tilde W_n\le z}  &= -D_x (\Indi{\tilde W_n\le z}),
\end{align*}
leading to the bound
\begin{align*}
|\ep_{n,2}|
  &\le \int_{\X} |\rho_n(x)| \EE[D_x \tilde W_n D_x( \tilde Wf(\tilde W) - \Indi{\tilde W_n\le z} )] \la(dx) = \ep_{n,2,1} + \ep_{n,2,2}
\end{align*}
where
\begin{align*}
\ep_{n,2,1}
  &:= \int_{\X} |\rho_n(x)| \rho_n(x) \EE[ D_x( \tilde W_n f(\tilde W_n) - \Indi{\tilde W_n\le z} )] \la(dx), \\
\ep_{n,2,2}
  &:= \int_{\X} |\rho_n(x)| \EE[ (D_x \tilde W_n -\rho_n(x)) D_x( \tilde W_n f(\tilde W_n) - \Indi{\tilde W_n\le z} )] \la(dx).
\end{align*}
Applying Lemma \ref{lem:tildegIBP} gives $$\ep_{n,2,1} =  \EE[\tilde\eta(|\rho_n|\rho_n) (\tilde Wf(\tilde W) - \Indi{\tilde W_n\le z}) ],$$ so that by  the fact that $|xf(x)|\le 1$ for any $x\in\RR$ in Lemma \ref{lem:steinsol}, one has
\begin{align}\label{eq:epsn21}
|\ep_{n,2,1}|
  &\le 3\EE[|\tilde \eta(|\rho_n|\rho_n)|] \le 3\Var[\tilde \eta(|\rho_n| \rho_n)]^{1/2}
	 = 3\Big(\int_{\X} |\rho_n(x)|^4 \la(dx)\Big)^{1/2}\nonumber\\
  &= \frac{3}{\sigma_n^2}  \Big( \sum_{p\in\Pc_n} \sum_{k\ge 1} \frac{k^3|\psi(p)|^4}{p^k} \Big)^{1/2}
	\le \frac{3c_1}{\sigma_n^2} \Big( 53\sum_{p\in\Pc_n} \frac{\psi(p)^2}{p}\Big)^{1/2}\nonumber\\
	&\leq \frac{159c_1}{\sigma_n^2} \sum_{p\in\Pc_n} \frac{\psi(p)^2}{p}(1-p^{-1})^{-1}\leq\frac{159c_1}{\sigma_n}.
\end{align}
On the other hand, using  $|xf(x)-yf(y)|\le 2$ and $|\1(x\le z)-\1(y\le z)|\le 1$ for any $x,y\in\RR$, we have
\begin{align*}
|\ep_{n,2,2}|&\le 3 \int_{\X} |\rho_n(x)| \EE[|D_x \tilde W_n -\rho_n(x)|]\la(dx).
\end{align*}
By \eqref{e:ep_n3}, this integral can be handled the same way as $\ep_{n,3}$, leading to the inequality
\begin{align}\label{eq:epsn22}
|\ep_{n,2,2}|
  &\le \frac{1}{\sigma_n^2}(16.8c_1^2+12.3c_1c_2)
\end{align}
Hence, by \eqref{eq:epsn21} and \eqref{eq:epsn22}
\begin{align*}
|\ep_{n,2,2}|\le  \frac{1}{\sigma_n^2} ( 49.5 c_1^2 + 24.6 c_1c_2 )
\end{align*}
so that
\begin{align}\label{e:ep_n2final}
|\ep_{n,2}|
  \le   \frac{1}{\sigma_n^2} ( 175.8c_1^2 + 24.6 c_1c_2 ).
\end{align}
Relations \eqref{eq:Tntotarget}, \eqref{e:ep_n3final} and \eqref{e:ep_n2final} lead to
\begin{align}\label{eq:e:5.1ppp}
|T_n - \EE[f'(\tilde W_n)I_n]
  &\leq \frac{1}{\sigma_n^2}(181.4c_1^2+28.7c_1c_2).
\end{align}
Relations
\eqref{e:5.1},\eqref{e:5.1p}, \eqref{e:5.1pp}, \eqref{e:ep_n1final} and \eqref{eq:e:5.1ppp} lead to the desired Kolmogorov bound.


\section{Proof of Theorem \ref{thm:mainunif}}\label{sec:mainunifthm}
Combining Theorem \ref{thm:CLTHarmonic} with Lemma \ref{lem:arra}, we deduce that
\begin{align}\label{eq:mainunifprev1}
\dk\left(\frac{\nphi(J_{n})-\mu_n}{\sigma_n},\Gauss\right)
  &\leq \dk\left(\frac{\nphi(H_{n}Q(n/H_n))-\mu_n}{\sigma_n},\frac{\nphi(H_{n})-\mu_n}{\sigma_n}\right)\\
	&+\frac{\gamma_1}{\sigma_n} + \frac{\gamma_2}{\sigma_n^2}+61\frac{\log\log(n)}{\log(n)},\nonumber
\end{align}
where $\Gauss$ is a random variable with standard Gaussian distribution and $\gamma_1,\gamma_2$, are given as in \eqref{eq:cidef}. Define
\begin{align*}
T_n
  &:=\Pb\left[\frac{\nphi(H_{n}Q(n/H_n))-\mu_n}{\sigma_n}\leq x\right]
  -\Pb\left[\frac{\nphi(H_{n})+\nphi(Q(n/H_n))-\mu_n}{\sigma_n}\leq x\right],
\end{align*}
and notice that
\begin{align*}
|T_n|
  &\leq |\Pb\left[\frac{\nphi(H_{n}Q(n/H_n))-\mu_n}{\sigma_n}\leq x \text{ and } Q(n/H_n)\not{|}\ H_{n}\right]\\
  &-\Pb\left[\frac{\nphi(H_{n})+\nphi(Q(n/H_n))-\mu_n}{\sigma_n}\leq x \text{ and } Q(n/H_n)\not{|}\ H_{n}\right]|
	+2 \Pb[Q(n/H_n) \text{ divides } H_{n}].
\end{align*}
By the additivity of $\psi$, the absolute value in the right is equal to zero, and thus,
\begin{align*}
|T_{n}|
  &\leq
	\frac{6.4\log\log(n)}{\log(n)},
\end{align*}
where the last inequality follows from Lemma \ref{lem:arra2}. Thus, by condition \textbf{(H1)}, for every $x\in\R$,
\begin{multline*}
\bigg{|}\Pb\left[\frac{\nphi(H_{n}Q(n/H_n))-\mu_n}{\sigma_n}\leq x\right]
-\Pb\left[\frac{\nphi(H_{n})-\mu_n}{\sigma_n}\leq x\right]\bigg{|}\\
\begin{aligned}
  &\leq \Pb\left[\frac{\nphi(H_{n})-\mu_n}{\sigma_n}\in[x,x+\sigma_n^{-1}c_1]\right]\vee \Pb\left[\frac{\nphi(H_{n})-\mu_n}{\sigma_n}\in[x-\sigma_n^{-1}c_1,x]\right]\\
	&\quad\quad+ \frac{6.4\log\log(n)}{\log(n)}.
\end{aligned}
\end{multline*}
By a further application of Theorem \ref{thm:CLTHarmonic},
\begin{multline}\label{eq:mainunifprev2}
\bigg{|}\Pb\left[\frac{\nphi(H_{n}Q(n/H_n))-\mu_n}{\sigma_n}\leq x\right]
-\Pb\left[\frac{\nphi(H_{n})-\mu_n}{\sigma_n}\leq x\right]\bigg{|}\\
\begin{aligned}
  &\leq \frac{6.4\log\log(n)}{\log(n)}+\sup_{x\in\R}\Pb\left[\Gauss\in[x,x+\sigma_n^{-1}c_1]\right]
	+\frac{\gamma_1}{\sigma_n}+\frac{\gamma_2}{\sigma_n^2}\\
	&\leq \frac{6.4\log\log(n)}{\log(n)}+\frac{1}{\sqrt{2\pi}\sigma_n}c_1+\frac{\gamma_1}{\sigma_n} + \frac{\gamma_2}{\sigma_n^2},
\end{aligned}
\end{multline}
where the last inequality follows from an application of the mean value theorem to the standard Gaussian density.
The Kolmogorov bound \eqref{eq:Kolmogorov_Jn} then follows from \eqref{eq:mainunifprev1} and \eqref{eq:mainunifprev2}.\\

\noindent To show \eqref{eq:Wasserstein_Jn}, we use Theorem \ref{thm:CLTHarmonic}, to get	
\begin{align}\label{ineq:d1finalb1mainthm}
d_1\left(\frac{\nphi(J_{n})-\mu_n}{\sigma_n},\Gauss\right)
  &\leq d_1\left(\frac{\nphi(H_{n})+\nphi(Q(n/H_n))-\mu_n}{\sigma_n},\frac{\nphi(H_{n})-\mu_n}{\sigma_n}\right)\\
	&+d_1\left(\frac{\nphi(H_{n})+\nphi(Q(n/H_n))-\mu_n}{\sigma_n},\frac{\nphi(J_n)-\mu_n}{\sigma_n}\right)
	+\frac{\gamma_3}{\sigma_n},\nonumber
\end{align}
where $\gamma_3$ is given by \eqref{eq:cidef}. The first term can be bounded in the following way
\begin{align}\label{ineq:d1finalb2mainthm}
d_1\left(\frac{\nphi(H_{n})+\nphi(Q(n/H_n))-\mu_n}{\sigma_n},\frac{\nphi(H_{n})-\mu_n}{\sigma_n}\right)
  &\leq \frac{1}{\sigma_n}\E[| \nphi(Q(n/H_n))|]
	\leq\frac{c_1}{\sigma_n}.
\end{align}
To bound the term
\begin{align*}
T_n
  &:=d_1\left(\frac{\nphi(H_{n})+\nphi(Q(n/H_n))-\mu_n}{\sigma_n},\frac{\nphi(J_n)-\mu_n}{\sigma_n}\right),
\end{align*}
we use Dobrushin's theorem, to guarantee the existence of a random variable $\mathfrak{Y}_n$ such that
$ \mathfrak{Y}_n \stackrel{Law}{=}\nphi(H_{n})+\nphi(Q(n/H_n))$ and
$$\dtv(\psi(J_n),\nphi(H_{n})+\nphi(Q(n/H_n)))=\Pb[\psi(J_n)\neq\mathfrak{Y}_n].$$
In principle, $\mathfrak{Y}_n$ might not be defined
over $(\Omega,\Fc,\Pb)$, but we will assume that it is, at the cost of extending the underlying probability space, if necessary. On the other hand,  for any random variables $X,Y$, the bound $\dw(X,Y)\le \EE[|\wt X-\wt Y|]$  holds where $\wt X$ and $\wt Y$ are equal in law to $X$ and $Y$ respectively. From here it follows that
\begin{align*}
T_n
  &\leq\frac{1}{\sigma_n}\E[|\nphi(J_n)- \mathfrak{Y}_n |]\nonumber
	=\frac{1}{\sigma_n}\E[|\nphi(J_n)-\mathfrak{Y}_n|\Indi{\mathfrak{Y}_n\neq \psi(J_n)}]\\
	&\leq \frac{\Pb[\mathfrak{Y}_n \neq \psi(J_n)]^{\frac{1}{2}}}{\sigma_n}(\|\nphi(J_n)\|_{L^2(\Omss)}
	+\|\nphi(H_{n})+\nphi(Q(n/H_n))\|_{L^2(\Omss)}),
\end{align*}
where the last step follows from Cauchy-Schwarz inequality. By Lemma \ref{lem:arra},
\begin{align*}
\Pb[\mathfrak{Y}_n\neq \psi(J_n)]
  &=\dtv(\nphi(H_{n})+\nphi(Q(n/H_n)),\nphi(J_n))\\
	&\leq \dtv(\nphi(H_{n}Q(n/H_n)),\nphi(J_n)) + \dtv(\nphi(H_{n})+\nphi(Q(n/H_n)),\nphi(H_{n}Q(n/H_n)))\\
  &\leq 61\frac{\log\log(n)}{\log\log(n)}+\dtv(\nphi(H_{n})+\nphi(Q(n/H_n)),\nphi(H_{n}Q(n/H_n))).
\end{align*}
To handle the remaining term, we observe that
\begin{align*}
\dtv(\nphi(H_{n})+\nphi(Q(n/H_n)),\nphi(H_{n}Q(n/H_n)))
  &\leq \Pb[Q(n/H_n)\text{ divides } H_{n}]\leq\frac{6.4\log\log(n)}{\log(n)},
\end{align*}
so that
\begin{align*}
\Pb[\mathfrak{Y}_n\neq \psi(J_n)]
  &\leq \frac{67.4\log\log(n)}{\log(n)}.
\end{align*}
From the previous analysis it follows that
\begin{align}\label{ineq:d3finalb1mainthm}
T_n
  &\leq 8.3\bigg(\frac{\log\log n}{\log n}\bigg)^{\frac{1}{2}}
	(\sigma_n^{-1}\|\nphi(J_n)\|_{L^2(\Omss)}+\sigma_n^{-1}\|\nphi(H_n)\|_{L^2(\Omss)}+\sigma_n^{-1 }\|\nphi\|_{\Pc }).
\end{align}
Combining \eqref{ineq:d3finalb1mainthm} with Lemma \ref{Lem:L2bounds}, we obtain
\begin{align*}
T_n
  &\leq \frac{\log\log (n)^{\frac{3}{2}}}{\log (n)^{\frac{1}{2}}}
	(24.9\sigma_n^{-1}c_1+60.4\sigma_n^{-1}c_2),
\end{align*}
which by the condition $\sigma_{n}^2\geq 3(c_1^2+c_2^2)$, gives
\begin{align*}
T_n
  &\leq 49.3\frac{\log\log (n)^{\frac{3}{2}}}{\log (n)^{\frac{1}{2}}}.
\end{align*}
Relation \eqref{eq:Wasserstein_Jn} follows from \eqref{ineq:d1finalb1mainthm}, \eqref{ineq:d1finalb2mainthm} and
\eqref{ineq:d3finalb1mainthm}.

\section{Proof of Theorem \ref{t:poisson}}\label{sec:theorempoisson}
\subsection{Harmonic case}
Recall that $\Lc(\psi(H_n))= \Lc(Y_n+R_n|\mathscr{A}_n)$, where $Y_n$ and $R_n$ are given by \eqref{eq:YnRndecompq}. Set $\tilde Y_n=Y_n+R_n$. In this section, we show that the Poisson approximation occurs for the non-standardized $\psi(H_n)$ as long as $\psi$ is equal to 1 for a sufficiently large proportion of $p\in\mathcal P$.

\medskip

\noindent Suppose that $h:\N_{0}\rightarrow\RR$ is an indicator function of a subset of $\N_0$. Let $f$ be the solution to \eqref{eq:steinpoissoneq} with
intensity $\lambda_{n}$, namely,
\begin{align*}
h(k) - \E[h(M_n)]
  &=\lambda_nf(k+1)-kf(k).
\end{align*}
Notice that
\begin{align*}
|\EE[h(\psi(H_n))] -\E[h(M_n)]|
  &=|\E[\lambda_nf(\psi(H_n)+1)-\psi(H_n) f(\psi(H_n))]|\\
  &=\Pb[\scr{A}_n]^{-1}|\E[(\lambda_nf(\tilde{Y}_n+1) - \tilde{Y}_nf(\tilde{Y}_n))I_n]|\\
	&\leq 2|\E[(\lambda_nf(\tilde{Y}_n+1) - (Y_n+R_n)f(\tilde{Y}_n))I_n]|.
\end{align*}
Using Lemma \ref{lem:steinsolPo} and the bound \eqref{eq:Rnmainbound} for $\EE[|R_n|]$, we get
\begin{align}\label{e:PLT_s1}
|\EE[h(\psi(H_n))] -\E[h(M_n)]|
  &\leq 2|\lambda_n\E[f(\tilde{Y}_n+1)I_n] - \E[Y_nf(\tilde{Y}_n))I_n]|+\lambda_n^{-\frac{1}{2}}(c_1 + 2c_2).
\end{align}
Define the kernel $\varrho_n:\mathbb{X}\rightarrow\R$ by
\begin{align*}
\varrho_n(p,k)
  &:=k\nphi(p) \1(p\le n)
\end{align*}
for $p\in\Pc$ and $k\in\N$.
Recall that by \eqref{eq:xipfuncrep}, $Y_n = \eta(\varrho_n)$ and $\EE[\eta(\varrho_n)]= \la_n$ by our choice of $\la_n$. Thus, by Mecke's equation \cite[Theorem 4.1]{LP} and \eqref{e:xi_add_one}
\begin{align*}
\E[Y_nf(\tilde{Y}_n)I_n]
&=\E[\eta(\varrho_n)f(\tilde{Y}_n)I_n] = \int \rho_n(x) \EE [f(\tilde Y_n(\eta+\de_x)) I_n(\eta+\de_x)] \la(dx)
\end{align*}
where for each $x=(p,k)\in\Pc_n\times \NN$, we have
\begin{align}
I_n(\eta+\de_x) &= \1(p^k \prod_{q\in \Pc_n\setminus\{p\}}q^{\xi_q}\le n), \notag\\
\tilde Y_n(\eta+\de_x)&= \psi(p^{k+\xi_p}) + \sum_{q\in \Pc_n\setminus\{p\}} \psi(p^{\xi_q}). \label{eq:tildeY_n_addone}
\end{align}
Therefore,
\begin{align}\label{e:PLT_s2}
\E[Y_nf(\tilde{Y}_n)I_n] -\la_n  \EE[f(\tilde Y_n + 1)] = \epsilon_{n,1} + \epsilon_{n,2},
\end{align}
where
\begin{align*}
\epsilon_{n,1} &:= \int \varrho_n(x) \EE[ (f(\tilde Y_n(\eta+\de_x)) - f(\tilde Y_n+1))  I_n(\eta+\de_x)] \la(dx), \\
\epsilon_{n,2}&:=\int \varrho_n(x) \EE [f(\tilde Y_n + 1) (I_n(\eta+\de_x) - I_n)] \la(dx).
\end{align*}
We have by the first half of Lemma \ref{lem:steinsolPo}  and \eqref{eq:ep_n_step3}, \eqref{eq:epsilonnbound} that
\begin{align}\label{e:PLT_s3}
|\epsilon_{n,2}|\le \frac{1}{\sqrt{\la_n}}\int |\varrho_n(x)| \EE[ |D_x I_n| ] \la(dx)  \le \frac{8c_1}{\sqrt{\la_n}}.
\end{align}
By the second half of Lemma \ref{lem:steinsolPo} and \eqref{eq:tildeY_n_addone},
\begin{align*}
|\epsilon_{n,1}| &\le \frac{1}{\la_n}\int |\varrho_n(x)| \EE[|\tilde Y_n(\eta+\de_x) - (\tilde Y_n+1)|] \la(dx) \\
&= \frac{1}{\la_n}\sum_{p\in\Pc_n}\sum_{k\ge 1} \frac{|\psi(p)|}{p^k}  \EE[|\psi(p^{\xi_p+k}) - \psi(p^{\xi_p})-1|] \\
&=\frac{1}{\la_n} \sum_{p\in\Pc_n} \frac{1-p^{-1}}{p}|\psi(p)| |\psi(p)-1|  +  \frac{1}{\la_n}\sum_{p\in\Pc_n} \frac{|\psi(p)|}{p} \EE[|\psi(p^{\xi_p+1})  - \psi(p^{\xi_p})-1|\1(\xi_p\ge 1)]  \\
&\quad + \frac{1}{\la_n}\sum_{p\in\Pc_n}\sum_{k\ge 2} \frac{|\psi(p)|}{p^k}  \EE[|\psi(p^{\xi_p+k}) -  \psi(p^{\xi_p})-1|]  =: \epsilon_{n,1,1} + \epsilon_{n,1,2} + \epsilon_{n,1,3}.
\end{align*}
This decomposition is totally parallel to our way of obtaining \eqref{e:ep_n3final}. Recalling \eqref{eq:delta's}, \eqref{eq:delta1finalb}
and \eqref{eq:delta2finalb}, we have
\begin{align*}
\epsilon_{n,1,2}
&\le \frac{c_1}{\la_n}
\big(\de_1 + \de_2 + \sum_{p\in\Pc_n} p^{-1}\PP[\xi_p\ge 1]\big)\\
  &\le \frac{c_1}{\la_n}[\zetab(2)^{1/2} c_2 + \zetab(2) c_1 + \zetab(4)^{1/2} c_2 + \zetab(2)]
  \leq \frac{1}{\lambda_n}(0.9c_1^2+1.1c_1c_2+0.7c_1).
\end{align*}
Similarly, recalling \eqref{eq:deltaprime's}, \eqref{eq:delta1primebound} and \eqref{eq:delta2primebound}, we have
\begin{align*}
\epsilon_{n,1,3}
  &\le \frac{c_1}{\la_n} (\de_1'+\de'_2
+ \sum_{p\in\Pc_n}\sum_{k\ge 2} \frac{1}{p^k})
\le \frac{c_1}{\la_n}  ( (11 \zetab(2))^{1/2} c_2 +   \zetab(3) c_1 + 2\zetab(6)^{1/2} c_2  + 2\zetab(2))\\
  &\le \frac{1}{\la_n}  (0.3c_1^2+3c_1c_2 +  1.3c_1),
\end{align*}
yielding
\begin{align}\label{e:PLT_s4}
|\epsilon_{n,1}|
  &\le  \frac{1}{\la_n} \sum_{p\in\Pc_n} \frac{|\psi(p)-1|}{p}
	+ \frac{1}{\la_n}  [1.2c_1^2 + 4.1c_1c_2 +  2c_1].
\end{align}
Combining \eqref{e:PLT_s1}, \eqref{e:PLT_s2}, \eqref{e:PLT_s3} and \eqref{e:PLT_s4} gives the desired bound.

\subsection{Uniform case: the total variation bound}
We first prove the total variation bound under the additional assumption that $\psi(p)=1$ for all $p\in\Pc$. For the Kolmogorov bound without this assumption, we follow the same strategy and prove it in the next subsection.

By the triangle inequality,
\begin{align}
\dtv(\psi(J_n), M_n)
&\leq \eta_1 + \eta_2 +\eta_3 + \eta_4. \label{e:gadef}
\end{align}
where
\begin{align*}
\begin{array}{ll}
\eta_1:=\dtv(\psi(J_n), \psi(H_nQ(n/H_n)))  &  \eta_2:=\dtv(\psi(H_nQ(n/H_n)),  \psi(H_n)+1)\\
\eta_3:=\dtv(\psi(H_n)+1, M+1)              &  \eta_4:=\dtv(M+1,M).
\end{array}
\end{align*}
Notice that Lemma \ref{lem:arra},
\begin{align}\label{eq:eta1bound}
\eta_1\le  61\frac{\log\log(n)}{\log(n)},
\end{align}
and the bound
\begin{align}\label{eq:eta3bound}
\eta_3
  &\leq \frac{\tilde{\gamma}_{1}}{\sqrt{\la_n}} +  \frac{\tilde{\gamma}_{2}}{\la_n} + 2c_1\sum_{p\in\Pc_n} \frac{|\psi(p)-1|}{p}
\end{align}
can be obtained from \eqref{eq:citildedef}.
%
For handling $\eta_4$, we use Stein's equation for the Poisson distribution with parameter $\la_n$.
By taking $M_n$ as the target distribution, we consider
\begin{align*}
&\la \EE[f(M_n+2)] -  \EE[(M_n+1)f(M_n+1)] \\
&= \la \EE[g(M_n+1)] - \EE[M_n g(M_n)] - \EE[g(M_n)]= -\EE[g(M)]= -\EE[f(M+1)],
\end{align*}
where we have used the condition $g(k)= f(k+1)$ for any $k\in \NN_0$, as well as the characterizing equation for the Poisson random variable $M_n$. Therefore, by the uniform bound in Lemma \ref{lem:steinsolPo}, one has
\begin{align}\label{eq:eta4bound}
\eta_4=\dtv(M_n+1, M_n) \le
\frac{1}{\sqrt{\la_n}}.
\end{align}

\noindent It remains to handle $\eta_2$. Define $\scr D_n =\{Q(n/H_n) \mbox{ divides } H_n\}$. Notice  the inclusion $$\{\psi(H_nQ(n/H_n))\neq  \psi(H_n)+1\}\subset \scr D_n.$$ By additivity, the assumption $\psi(p)=1$ and Lemma \ref{lem:arra2},
\begin{align*}
\eta_2\le \PP[\scr D_n] \le \frac{6.4\log\log(n)}{\log(n)}.
\end{align*}
The desired bound follows immediately.

\subsection{Uniform case: the Kolmogorov bound} Due to the relation $\dk(X,Y)\le \dtv(X,Y)$ for arbitrary random variables $X,Y$, we have the bound
\begin{align*}
\dk(\psi(J_n), M_n)
&\leq \eta_1 + \eta'_2 +\eta_3 + \eta_4,
\end{align*}
where
\begin{align*}
\eta'_2:= \dk(\psi(H_nQ(n/H_n)),  \psi(H_n)+1).
\end{align*}
It remains to handle $\eta'_2$.
Let $z\in\RR$ be given. Then,
\begin{align*}
T
  &:=|\PP[\psi(H_n Q(n/H_n))\le z] - \PP[\psi(H_n)+1\le z]| \\
  &\le |\PP[\psi(H_n) + \psi(Q(n/H_n))\le z, \scr D_n^c] - \PP[\psi(H_n)+1\le z, \scr D_n^c ]| + \PP[\scr D_n].
\end{align*}
Let $A\Delta B$ denote the symmetric difference of two given subsets $A,B\in\Omss$.
Observe that the first term in the right-hand side is bounded by
\begin{multline*}
\PP[\{\psi(H_n) + \psi(Q(n/H_n))\le z, \scr D_n^c\}\Delta\{\psi(H_n)+1\le z, \scr D_n^c \}]\\
\begin{aligned}
  &\leq \PP[\{\psi(H_n) + \psi(Q(n/H_n))\le z\}\backslash\{\psi(H_n)+1\le z\}]\\
	&+\PP[\{\psi(H_n)+1\le z\}\backslash\{\psi(H_n) + \psi(Q(n/H_n))\le z\}].
\end{aligned}
\end{multline*}
Moreover, by \textbf{(H1)}, we have the inclusions
\begin{align*}
\{\psi(H_n) + \psi(Q(n/H_n))\le z\}\backslash\{\psi(H_n)+1\le z\}
\subset\{ \psi(H_n)\in [z-1,z+c_1\vee 1] \}\\
\{\psi(H_n)+1\le z\}\backslash\{\psi(H_n) + \psi(Q(n/H_n))\le z\}
\subset\{ \psi(H_n)\in [z-c_1\vee 1	,z+1] \},
\end{align*}
and consequently,
\begin{align*}
T
  &\le 2\PP[ z-c_1\vee 1\le \psi(H_n) \le z+c_1\vee 1] + \PP[\scr D_n] \\
  &\le 2 \dtv(\psi(H_n), M_n) + 2\PP[  z-c_1\vee 1\le M_n \le z+c_1\vee 1] + \PP[\scr D_n].
\end{align*}
Combining Theorem \ref{thm:PLTHarmonic} with Lemma \ref{lem:arra2}, we thus obtain the bound
\begin{align*}
T
  &\le 2\PP[ z-c_1\vee 1\le M_n \le z+c_1\vee 1]+\frac{2\tilde{\gamma}_{1}}{\sqrt{\la_n}} +  \frac{2\tilde{\gamma}_{2}}{\la_n} + 4c_1\sum_{p\in\Pc_n} \frac{|\psi(p)-1|}{p} + 6.4\frac{\log\log(n)}{\log(n)}.
\end{align*}
Finally, using the fact that the Poisson distribution is unimodal (which implies that the probability of the atoms of $M_n$ is
bounded by $\frac{\lambda_n^{\lambda_n}}{\lambda_n!}e^{-\lambda_n}$), as well as Stirling's formula, we get the estimate
\begin{align*}
\PP[ z-c_1\vee 1\le M_n \le z+c_1\vee 1]
  &\leq \frac{2(c_1\vee 1)}{\sqrt{2\pi\lambda_n}}
\end{align*}
From here we conclude that
\begin{align}\label{eq:eta2bound}
\eta_2
  &\leq \frac{2\tilde{\gamma}_{1}}{\sqrt{\la_n}} +  \frac{1}{\la_n}\big(2\tilde{\gamma}_{2}+\frac{4(c_1\vee 1)}{\sqrt{2\pi}}\big) + 4c_1\sum_{p\in\Pc_n} \frac{|\psi(p)-1|}{p} + 6.4\frac{\log\log(n)}{\log(n)}.
\end{align}
Relation \eqref{eq:t:poisson} follows from \eqref{eq:eta1bound}-\eqref{eq:eta2bound}.


\appendix

\section{Generalization of Theorem \ref{thm:multiplicities}}\label{sec:generalKeystep}
Next we present an extension of Proposition \ref{thm:multiplicities}. Although Proposition \ref{thm:multiplicities} is good enough for proving
the main results of the manuscript, Proposition \ref{thm:multiplicities2} below illustrates that not only the law of $H_n$ satisfies
a relation of the type \eqref{eq:conditionallaws}, but also any random variable supported in $\N$ with a suitable multiplicative
property over its probability distribution.
\begin{prop}\label{thm:multiplicities2}
Suppose that $n\geq 21$. Let $\vartheta:\N\rightarrow\R_{+}$ be a non-negative multiplicative function (i.e. $\vartheta(mn)=\vartheta(m)\vartheta(n)$ if
$m$ and $n$ are co-prime) satisfying $\sum_{k=0}^{\infty}\vartheta(p^k)<\infty$ for all $p\in\Pc$. Let $H_n^{\vartheta}$ be a random variable defined in
$(\Omss,\Fc,\Pb)$ and supported in $\N\cap[1,n]$,  with probability distribution given by
\begin{align*}
\Pb[H_{n}^{\vartheta}=k]
  &=\frac{1}{L_n^{\vartheta}}\vartheta(k),
\end{align*}
for $k=1,\dots, n$ and $L_{n}^{\vartheta}:=\sum_{k=1}^n\vartheta(k)$. Consider a family $\{\xi_{p}^{\vartheta}\}_{p\in\Pc}$ of independent random variables defined in
$(\Omss, \Fc,\Pb)$, with
\begin{align*}
\Pb[\xi_p^{\vartheta}=k]
  &=\nu_p\vartheta(p^k),
\end{align*}
for some $\nu_p\geq0$ satisfying $\sum_{k=0}^{\infty}\nu_p\vartheta(p^k)=1$. Define the event
\begin{align*}
A_{n}^{\vartheta}
  &:=\Big\{\prod_{p\in\mathcal{P}_n}p^{\xi_{p}^{\vartheta}}\leq n\Big\},
\end{align*}
as well as the random vector $\vec{C}_{\vartheta}(n):=(\alpha_{p}(H_n^{\vartheta}) ; p\in\Pc_n)$. Then
\begin{align}\label{eq:PorbAsn2}
\PP(A_n^{\vartheta})
  &=L_n^{\vartheta}\prod_{p\in\Pc_n}\nu_p,
\end{align}	
and
\begin{align}\label{eq:conditionallaw2}
    \Lc(\vec{C}_{\vartheta}(n))
      &=\Lc(\vec{\xi}^{\vartheta}(n)\ |\ A_n^{\vartheta}),
\end{align}
where $\vec{\xi}^{\vartheta}(n):=(\xi_p^{\vartheta} ; p\in\Pc_n)$.
\end{prop}
\begin{Remark}
By choosing $\nu_{p}=(1-p^{-1})$ and $\vartheta(m)=\frac{1}{m}$, we obtain Proposition \ref{thm:multiplicities} as a Corollary of Proposition \ref{thm:multiplicities2}
\end{Remark}
\begin{proof}
Consider a fixed vector $\vec{\lambda}=(\lambda_p ; p\in\Pc_n)\in\R^{\pi(n)}$ and define $\mathcal{K}_{n}$ by \eqref{eq:mathKndef}.
As in the proof of Proposition \ref{thm:multiplicities2}, the prime factorization theorem allows us to write
\begin{align*}
\E[e^{\sum_{p\in\Pc_n}\textbf{i} \lambda_p\xi_p^{\vartheta}} \Indi{A_n}]
  &= \sum_{(c_p;\ p\in\Pc_n)\in\mathcal{K}_n}
	\exp\{\textbf{i}\sum_{p\in\Pc_n} \la_p c_p \} \Pb[\xi_p^{\vartheta}=c_p \text{ for all }p\in\Pc_n] \\
  &= \sum_{(c_p;\ p\in\Pc_n)\in\mathcal{K}_n}  \exp\{\textbf{i} \sum_{p\in\Pc_n}\la_p \al_p(\prod_{\theta\in\mathcal P_n} \theta^{c_{\theta}}) \}
	\Pb[\xi_p^{\vartheta}=c_p \text{ for all }p\in\Pc_n] \\
&=  \sum_{(c_p;\ p\in\Pc_n)\in\mathcal{K}_n}  \exp\{\textbf{i} \sum_{p\in\Pc_n}\la_p \al_p(\prod_{\theta\in\mathcal P_n} \theta^{c_{\theta}}) \}
\big(\prod_{p\in\Pc_n} \nu_p\big)\big( \prod_{p\in\Pc_n} \vartheta(p^{c_p})\big)\\
&=  \sum_{(c_p;\ p\in\Pc_n)\in\mathcal{K}_n}  \exp\{\textbf{i} \sum_{p\in\Pc_n}\la_p \al_p(\prod_{\theta\in\mathcal P_n} \theta^{c_{\theta}}) \}
\big(\prod_{p\in\Pc_n} \nu_p\big) \vartheta(\prod_{p\in\Pc_n}p^{c_p})\\
&= \sum_{k=1}^n\exp\{\textbf{i} \sum_{p\in\Pc_n} \la_p \al_p(k)\} \vartheta(k)\prod_{p\in\Pc_n}\nu_p
=\E[\exp\{\sum_{p\in\Pc_n}\textbf{i} \lambda_p\alpha_p(H_n^{\vartheta})\}]L_n^{\vartheta}\prod_{p\in\Pc_n}\nu_p.
\end{align*}
Relations \eqref{eq:PorbAsn2} and \eqref{eq:conditionallaw2} then follow analogously to the proof of Proposition \ref{thm:multiplicities}
\end{proof}

\section{Technical lemmas}
\noindent In this section, we prove some technical lemmas that were repidetely used throughout the manuscript.
\begin{Lemma}\label{Lemma:variancesomega}
Let $\omega, \Omega$ be the prime counting functions defined by \eqref{eq:littleomegadef} and \eqref{eq:bigomegadef}. Then, for all $n\geq 21,$
\begin{align}
|\E[\omega(J_n)]-\log\log(n)|
  &\leq 1.5,\label{eq:Eomega}\\
|\EE[\Omega(J_n)]-\log\log(n)|
  &\leq 4.8,\label{eq:EOmega}\\
\E[\omega(J_n)^2]
  &\leq 5.4\log\log(n)^2.\label{eq:Eomegasqu}\\
\E[\omega(H_n)^2]
  &\leq 5.4\log\log(n)^2.\label{eq:EomegasquHar}
\end{align}
\end{Lemma}
\begin{proof}
By using the representation
\begin{align*}
\omega(J_{n})
  &=\sum_{p\in\Pc_n}\Indi{\alpha_p(J_n)\geq 1},
\end{align*}
as well as identity \eqref{eq:probdivide}, we can write
\begin{align}\label{eq:Eomegaprev}
|\E[\omega(J_{n})]-\sum_{p\in\Pc_n}\frac{1}{p}|
  &\leq\sum_{p\in\Pc_n}|\Pb[p \text{ divides } J_n]-\frac{1}{p}|
	\leq\frac{|\Pc_n|}{n}\leq \frac{1.5}{\log(n)},
\end{align}
where the last inequality follows from the fact that $\pi(n)\leq 1.5\frac{n}{\log(n)}$. Relation \eqref{eq:Eomega},  follows from
\eqref{eq:Eomegaprev}, \eqref{eq:Mertensinv} and the condition $n\geq 21$.\\

\noindent To prove \eqref{eq:EOmega}, we notice that
\begin{align*}
\EE[\Omega(J_n)] = \sum_{p\in\Pc_n} \EE[\al_p(J_n)] = \sum_{p\in\Pc_n}\sum_{k=1}^\infty \PP(\al_p(J_n)\ge k) =  \sum_{p\in\Pc_n}\sum_{k=1}^\infty \PP(p^k\mbox{ divides } J_n).
\end{align*}
Hence, using \eqref{eq:probdivide}, \eqref{eq:Mertensinv} and the second inequality in \eqref{eq:Eomegaprev}, we have
\begin{align*}
|\EE[\Omega(J_n)] - \log\log(n)|&\le 0.262 + \frac{3.5}{\log n}  +  \sum_{p\in\Pc_n} \sum_{k\ge 2} \PP(p^k \mbox{ divides } J_n) \\
&\le 0.262 + \frac{3.5}{\log n}  + 2\zetab(2) \le 4.8,
\end{align*}
where the last inequality follows from the condition $n\geq 21$.
\noindent To show \eqref{eq:Eomegasqu}, we notice that by \eqref{eq:probdivide},
\begin{align*}
\E[\omega(J_{n})^2]
  &=\sum_{p,q\in\Pc_n}\Pb[p,q \text{ divide } J_{n}]
	=\sum_{\substack{p,q\in\Pc_n\\p\neq q}}\Pb[pq \text{ divides } J_{n}] + \sum_{p\in\Pc_n}\Pb[p \text{ divides } J_{n}].
\end{align*}
Therefore, by \eqref{eq:Mertensinv}, we conclude that
\begin{align}\label{eq:omegafbaux1}
\E[\omega(J_{n})^2]
  &\leq\sum_{\substack{p,q\in\Pc_n\\p\neq q}}\frac{1}{pq} + \sum_{p\in\Pc_n}\frac{1}{p}
	\leq \bigg(1+\sum_{p\in\Pc_n}\frac{1}{p}\bigg)\bigg(\sum_{p\in\Pc_n}\frac{1}{p}\bigg)\nonumber\\
	&\leq (\log\log(n)+2)(\log\log(n)+1)\leq \log\log(n)^2+3\log\log(n)+2\leq5.4\log\log(n)^2,
\end{align}
as required.\\

\noindent To show \eqref{eq:EomegasquHar}, we write
\begin{align*}
\E[\omega(H_{n})^2]
  &=\sum_{\substack{p,q\in\Pc_n\\p\neq q}}\Pb[pq \text{ divides } H_{n}] + \sum_{p\in\Pc_n}\Pb[p \text{ divides } H_{n}].
\end{align*}
We can easily show that for every $m\in\N$,
\begin{align}\label{eq:m|Hn}
\Pb[m \text{ divides } H_{n}]
  &=\frac{1}{L_n}\sum_{1\leq k\leq n}\frac{1}{k}\Indi{m \text{ divides } k}
	=\frac{1}{L_n}\sum_{1\leq j\leq n/m}\frac{1}{jm}\leq\frac{1}{m},
\end{align}
and thus, by \eqref{eq:omegafbaux1},
\begin{align*}
\E[\omega(H_{n})^2]
  &\leq\sum_{\substack{p,q\in\Pc_n\\p\neq q}}\frac{1}{pq} + \sum_{p\in\Pc_n}\frac{1}{p}
	\leq 5.4\log\log(n)^2,
\end{align*}
as required.
\end{proof}

\begin{Lemma}\label{Lem:L2bounds}
Let $\psi$ be a general additive function subject to  \textbf{(H1)} and \textbf{(H2)}. Then
\begin{align*}
\EE[\psi(H_n)^2] \le c_1^2  \log\log(n)^2 + 13.2 c_2^2. \\
\EE[\psi(J_n)^2] \le c_1^2  \log\log(n)^2 + 13.2 c_2^2. \\
\end{align*}
\end{Lemma}
\begin{proof}
We first write
\begin{align*}
\EE[\psi(H_n)^2] = \sum_{p\in\Pc_n} \EE[\psi(p^{\al_p(H_n)})^2] + \sum_{p\neq q\in\Pc_n} \EE[\psi(p^{\al_p(H_n)})\psi(p^{\al_p(H_n)})].
\end{align*}
Notice that by \eqref{eq:m|Hn} and $\PP(\xi_p=k)= (1-p^{-1})p^{-k}$, we have
\begin{align*}
\EE[\psi(p^{\al_p(H_n)})^2] &= \sum_{k=1}^\infty \PP(\al_p(H_n)=k) \psi(p^k)^2 \le\sum_{k=1}^\infty \PP(p^k \mbox{ divides } H_n)\psi(p^k)^2 \\
&\le 2\sum_{k=1}^\infty \PP(\xi_p = k)\psi(p^k)^2=2\EE[\psi(p^{\xi_p})^2].
\end{align*}
Similarly, we see that for $p\neq q$,
\begin{align*}
\EE[\psi(p^{\al_p(H_n)})\psi(p^{\al_p(H_n)})] &\le \sum_{k,\ell=1}^\infty \PP(\al_p(H_n)=k, \al_q(H_n)=\ell) |\psi(p^k)\psi(q^\ell)| \\
&\le \sum_{k,\ell=1}^\infty \PP( p^k q^\ell \mbox{ divides } H_n)  |\psi(p^k)\psi(q^\ell)| \\
&\le \Big(\sum_{k=1}^\infty \frac{\psi(p^k)}{p^k} \Big)   \Big(\sum_{k=1}^\infty \frac{\psi(q^k)}{q^k} \Big) \le 4 \EE[|\psi(p^{\xi_p})|]\EE[|\psi(q^{\xi_q})|].
\end{align*}
Therefore, we have
\begin{align*}
\EE[\psi(H_n)^2] \le 4 \Big(\sum_{p\in\Pc_n} \EE[|\psi(p^{\xi_p})|]\Big)^2.
\end{align*}
We notice that by using \eqref{eq:probdivide} in place of \eqref{eq:m|Hn} we also have
\begin{align*}
\EE[\psi(J_n)^2] \le 4 \Big(\sum_{p\in\Pc_n} \EE[|\psi(p^{\xi_p})|]\Big)^2.
\end{align*}
To bound each of the summands, we infer from the fact $\mathcal L(\xi_p|\xi_p\ge 2)=\mathcal L(2+\xi_p)$ that
\begin{align*}
\EE[|\psi(p^{\xi_p})|] &= (1-p^{-1})p^{-1}|\psi(p)| + p^{-2} \EE[|\psi(p^{2+\xi_p})|].
\end{align*}
One concludes that
\begin{align*}
\EE[\psi(H_n)^2] \vee \EE[\psi(H_n)^2] &\le 4 \Big( 0.5 c_1 \sum_{p\in\Pc_n} p^{-1} + \sum_{p\in\Pc_n} \frac{\Psi(p)}{p^2}\Big)^2 \\
&\le c_1^2  \log\log(n)^2 + 8\zetab(2) c_2^2,
\end{align*}
where we have used \eqref{eq:Mertensinv} for bounding the first series and Cauchy-Schwarz's inequality for the second. The proof is now complete.
\end{proof}

\begin{lemma}\label{l:reduction2}
\begin{align*}
\EE[\sum_{p\in\mathcal P} |\nphi(p) |\xi_p \mathbbm{1}(\xi_p\ge 2) ] < 3 \bar{\zeta}(2) c_1\leq 2c_1\\
\EE[\sum_{p\in\mathcal P_n}|\nphi(p^{\xi_p})| \mathbbm{1}(\xi_p\ge 2)]  \le \bar{\zeta}(2)^{1/2} c_2\leq c_2
\end{align*}
\end{lemma}
\begin{proof}
By the identity $\mathcal{L}(\xi_p\ |\ \xi_p\geq 2)=\mathcal{L}(\xi_p+2)$, we have
\begin{align*}
\EE[\sum_{p\in\mathcal P_n} |\nphi(p)| \xi_p \mathbbm{1}(\xi_p\ge 2) ]
&\le c_1 \sum_{p\in\mathcal P}  \PP[\xi_p\ge 2] \EE[2+\xi_p]\\
&=c_1 \sum_{p\in\mathcal P}  p^{-2}(2+p^{-1}(1-p^{-1})^{-1})
\le 3 \bar{\zeta}(2) c_1.
\end{align*}
by (H1). For the other term
\begin{align*}
\EE [\sum_{p\in\mathcal P_n}|\nphi(p^{\xi_p})| \mathbbm{1}(\xi_p\ge 2)]
&= \sum_{p\in\mathcal P_n} \EE[ |\nphi(p^{\xi_p+2})| ] \PP[\xi_p\ge 2]
=\sum_{p\in\mathcal P_n}  \frac{\EE[ |\nphi(p^{\xi_p+2}) |]}{p^2} \\
&\le \Big(\sum_{p\in\mathcal P} \frac 1 {p^2} \Big)^{1/2}
\Big( \sum_{p\in\mathcal P} \frac{\EE[| \nphi(p^{\xi_p+2}) |]^2}{p^2} \Big)^{1/2}
\le \bar{\zeta}(2)^{1/2} c_2
\end{align*}
by (H2).
\end{proof}

\begin{Lemma}\label{l:HnTail}
For any $\theta\in [n]$ and $n\ge 21$, we have
\begin{align}\label{ineq:Hntailestimate}
\PP[H_n > n\theta^{-1}]
  &\le \frac{2\log \theta}{L_n}.
\end{align}
\end{Lemma}
\begin{proof}
The inequality is trivial when $n\leq 2\theta$, as in such instance the right hand side is bounded from below by
$\frac{2\log (n/2)}{\log(n)+1}\geq1$, due to the condition $n\geq 21$. Thus, we can assume without loss of generality that $n>2\theta$. Notice that for all $k\geq 1$,
\begin{align*}
\sum_{i=k}^n \frac{1}{i} \le \int_{k-1}^{n} \frac{1}{x} dx = \log(n) - \log(k-1),
\end{align*}
yielding
\begin{align*}
\PP[H_n> n\theta^{-1}]
  &= \sum_{k=\floor{n\theta^{-1}}+1}^n \frac{1}{kL_n} \le \frac{1}{L_n} (\log(n) - \log(\floor{n\theta^{-1}}) ) \\
  &\le \frac{1}{L_n} \log\bigg(\frac{n}{n\theta^{-1}-1}\bigg) \le \frac{1}{L_n} \log(2\theta) \le \frac{2\log \theta}{L_n},
\end{align*}
where the one but last inequality follows from the fact that $n> 2\theta$.
\end{proof}

\begin{Lemma}\label{eq:Lemageometricaux}
For every $p\geq 2$, we have that
\begin{align}\label{eq:Lemageometricauxeq}
\sum_{k=1}^{\infty}kp^{-k}
  &=p^{-1}(1-p^{-1})^{-2}\leq \frac{2}{p}\\
\sum_{k=1}^{\infty}k^2p^{-k}
  &=p^{-1}(1-p^{-1})^{-3}(1+p^{-1})\leq \frac{12}{p}\label{eq:Lemageometricauxeq2}\\
\sum_{k=1}^{\infty}k^3p^{-k}
  &=p^{-3}(1+4p+p^2)(1-p^{-1})^{-4}
	\leq \frac{53}{p}.\label{eq:Lemageometricauxeq3}
\end{align}
\end{Lemma}

\begin{proof}
The result easily follows from the fact that if $G$ has geometric distribution with
$\Pb[G=k]=\theta(1-\theta)^{k-1}$, then its moment generating function is given by
\begin{align*}
\E[e^{\lambda G}]
  &=\frac{\theta}{1-(1-\theta)e^{\lambda}}.
\end{align*}
The result is thus obtained by multiplying the both sides of \eqref{eq:Lemageometricauxeq}-\eqref{eq:Lemageometricauxeq3} by $(1-p^{-1})$,
taking the first two derivatives in $\E[e^{\lambda G}]$ and evaluating at $\lambda=0$ and $\theta=1-p^{-1}$.
\end{proof}

\begin{Lemma}\label{l:GaussianDistance}
Let $N$ be a standard normal random variable and $W$ be normal with mean $\mu$ and variance $\sigma^2$. Then
\begin{align*}
\dk(W,N)\le |\sigma^2-1 |  +  \frac{\sqrt{2\pi}}{4}\mu, \\
\dw(W,N) \le \frac{\sqrt{2}}{\pi}|\sigma^2-1 | +  2\mu.
\end{align*}
\end{Lemma}
\begin{proof}
By integration by parts, one sees that
\begin{align*}
\sigma^2\EE[f'(W)] = \EE[(W-\mu)f(W)]
\end{align*}
for all $f:\RR\to\RR$ with $\|f\|_\infty + \|f'\|_\infty<\infty$, yielding
\begin{align*}
|\EE[Wf(W)] -\EE[f'(W)]|\le \|f'\|_\infty   |\sigma^2-1| + \|f\|_\infty \mu.
\end{align*}
The result follows from Lemmas \ref{lem:steinsol} and \ref{lem:steinsol_W}.

\end{proof}

\noindent{\bf Acknowledgments}. This research is supported by FNR Grant R-AGR-3410-12-Z (MISSILe) from the University of Luxembourg and partially supported by Grant R-146-000-230-114 from the National University of Singapore.


\begin{thebibliography}{10}

\bibitem{MRArratia}
R.~Arratia.
\newblock On the amount of dependence in the prime factorization of a uniform
  random integer.
\newblock In {\em Contemporary combinatorics}, volume~10 of {\em Bolyai Soc.
  Math. Stud.}, pages 29--91. J\'{a}nos Bolyai Math. Soc., Budapest, 2002.

\bibitem{BarVin}
M.~B. Barban and A.~I. Vinogradov.
\newblock On the number-theoretic basis of probabilistic number theory.
\newblock {\em Dokl. Akad. Nauk SSSR}, 154:495--496, 1964.

\bibitem{BaKoNi}
A.~D. Barbour, E.~Kowalski, and A.~Nikeghbali.
\newblock Mod-discrete expansions.
\newblock {\em Probab. Theory Related Fields}, 158(3-4):859--893, 2014.

\bibitem{BaHolst}
A.~D.~Holst Barbour and S.~Janson.
\newblock {\em Poisson approximation}.
\newblock Oxford studies in probability. Oxford, England, 1992.

\bibitem{Bill}
Patrick Billingsley.
\newblock On the central limit theorem for the prime divisor functions.
\newblock {\em Amer. Math. Monthly}, 76:132--139, 1969.

\bibitem{ChatSoDi}
Sourav Chatterjee, Persi Diaconis, and Elizabeth Meckes.
\newblock Exchangeable pairs and poisson approximation.
\newblock {\em Probability Surveys}, 2, 12 2004.

\bibitem{Chen}
Louis H.~Y. Chen.
\newblock Poisson approximation for dependent trials.
\newblock {\em Ann. Probability}, 3(3):534--545, 1975.

\bibitem{ChGoSh}
Louis H.~Y. Chen, Larry Goldstein, and Qi-Man Shao.
\newblock {\em Normal approximation by {S}tein's method}.
\newblock Probability and its Applications (New York). Springer, Heidelberg,
  2011.

\bibitem{DP18} D\"obler, Christian; Peccati, Giovanni. The fourth moment theorem on the Poisson space. Ann. Probab. 46 (2018), no. 4, 1878--1916.

\bibitem{Elliot}
P.~D. T.~A. Elliott.
\newblock {\em Probabilistic number theory. {I}}, volume 239 of {\em
  Grundlehren der Mathematischen Wissenschaften [Fundamental Principles of
  Mathematical Science]}.
\newblock Springer-Verlag, New York-Berlin, 1979.
\newblock Mean-value theorems.

\bibitem{ErdKac1}
P.~Erd\"os and M.~Kac.
\newblock The {G}aussian law of errors in the theory of additive number
  theoretic functions.
\newblock {\em Amer. J. Math.}, 62:738--742, 1940.

\bibitem{Erha}
Torkel Erhardsson.
\newblock {\em Stein's method for Poisson and compound Poisson approximation},
  pages 61--113.
\newblock 04 2005.

\bibitem{Harp}
Adam~J. Harper.
\newblock Two new proofs of the {E}rd\"os-{K}ac theorem, with bound on the rate
  of convergence, by {S}tein's method for distributional approximations.
\newblock {\em Math. Proc. Cambridge Philos. Soc.}, 147(1):95--114, 2009.

\bibitem{JaKoNi}
Jean Jacod, Emmanuel Kowalski, and Ashkan Nikeghbali.
\newblock Mod-{G}aussian convergence: new limit theorems in probability and
  number theory.
\newblock {\em Forum Math.}, 23(4):835--873, 2011.

\bibitem{KoNi}
Emmanuel Kowalski and Ashkan Nikeghbali.
\newblock Mod-{P}oisson convergence in probability and number theory.
\newblock {\em Int. Math. Res. Not. IMRN}, (18):3549--3587, 2010.

\bibitem{Kub}
J.~Kubilius.
\newblock {\em Probabilistic methods in the theory of numbers}.
\newblock Translations of Mathematical Monographs, Vol. 11. American
  Mathematical Society, Providence, R.I., 1964.

\bibitem{LrPY20+} Rapha\"el Lachieze-Rey, Giovanni Peccati, Xiaochuan Yang. Quantitative two-scale stabilisation on the Poisson space, preprint.


\bibitem{LP}  G\"unter Last and Mathew Penrose. Lectures on the Poisson process. Institute of Mathematical Statistics Textbooks, 7. Cambridge University Press, Cambridge, 2018. xx+293 pp.

\bibitem{LPS16} Last, Günter; Peccati, Giovanni; Schulte, Matthias. Normal approximation on Poisson spaces: Mehler's formula, second order Poincaré inequalities and stabilization. Probab. Theory Related Fields 165 (2016), no. 3-4, 667--723.

\bibitem{LeVeq}
Wm.~J. LeVeque.
\newblock On the size of certain number-theoretic functions.
\newblock {\em Trans. Amer. Math. Soc.}, 66:440--463, 1949.

\bibitem{NP} Ivan Nourdin and Giovanni Peccati. Normal approximations with Malliavin calculus. From Stein's method to universality. Cambridge Tracts in Mathematics, 192. Cambridge University Press, Cambridge, 2012. xiv+239 pp.

\bibitem{TurRen}
A.~R\'{e}nyi and P.~Tur\'{a}n.
\newblock On a theorem of {E}rd\"os-{K}ac.
\newblock {\em Acta Arith.}, 4:71--84, 1958.

\bibitem{Rosser}
J.~Barkley Rosser and Lowell Schoenfeld.
\newblock Approximate formulas for some functions of prime numbers.
\newblock {\em Illinois J. Math.}, 6:64--94, 1962.

\bibitem{Stein}
Charles Stein.
\newblock A bound for the error in the normal approximation to the distribution
  of a sum of dependent random variables.
\newblock In {\em Proceedings of the {S}ixth {B}erkeley {S}ymposium on
  {M}athematical {S}tatistics and {P}robability ({U}niv. {C}alifornia,
  {B}erkeley, {C}alif., 1970/1971), {V}ol. {II}: {P}robability theory}, pages
  583--602, 1972.

\bibitem{Ten2}
G\'{e}rald Tenenbaum.
\newblock Crible d'\'{E}ratosth\`ene et mod\`ele de {K}ubilius.
\newblock In {\em Number theory in progress, {V}ol. 2
  ({Z}akopane-{K}o\'{s}cielisko, 1997)}, pages 1099--1129. de Gruyter, Berlin,
  1999.

\bibitem{Ten}
G\'{e}rald Tenenbaum.
\newblock {\em Introduction to analytic and probabilistic number theory},
  volume 163 of {\em Graduate Studies in Mathematics}.
\newblock American Mathematical Society, Providence, RI, third edition, 2015.
\newblock Translated from the 2008 French edition by Patrick D. F. Ion.

\bibitem{Trud}
Tim Trudgian.
\newblock Updating the error term in the prime number theorem.
\newblock {\em Ramanujan J.}, 39(2):225--234, 2016.

\end{thebibliography}
\bibliographystyle{plain}

\end{document}